\documentclass[12pt]{scrartcl}

\newif\ifdetails\detailsfalse
\newif\ifprinttoc\printtocfalse
\InputIfFileExists{unstable.cfg}{}{}

\usepackage{amsmath,amsfonts,amssymb,amsthm,graphicx}
\usepackage{hyperref}
\hypersetup{unicode,bookmarksdepth=3}
\usepackage{nameref}
\usepackage[hypcap=true]{caption}
\usepackage[hypcap=true]{subcaption}
\usepackage{braket}
\usepackage{enumitem}
\newlist{assumptions}{enumerate}{1}
\setlist[assumptions]{
    label=\Alph*),
    ref=Assumption \theassumption.\Alph*),
    noitemsep
}

\usepackage{ifxetex,ifluatex}
\newif\ifunicode\unicodetrue
\ifxetex\else\ifluatex\else\unicodefalse\fi\fi

\ifunicode
\usepackage[colon=literal]{unicode-math}
\usepackage{polyglossia}
\setmainlanguage{english}
\setotherlanguages{russian}
\def\t{\tilde}
\newcommand{\wt}{\widetilde}
\newcommand{\trans}{⫛}
\newcommand{\bfC}{𝐂}
\newcommand{\bfT}{𝐓}

\newcommand{\bfV}{𝐕}
\newcommand{\bfM}{𝐌}

\newcommand{\mcI}{ℐ}
\newcommand{\mcH}{ℋ}
\newcommand{\mcV}{𝒱}
\else
\newcommand{\trans}{\pitchfork}
\newcommand{\bfC}{{\mathbf C}}
\newcommand{\bfT}{{\mathbf T}}

\newcommand{\bfV}{{\mathbf V}}
\newcommand{\bfM}{{\mathbf M}}

\newcommand{\mcI}{{\mathcal I}}
\newcommand{\mcH}{{\mathcal H}}
\newcommand{\mcV}{{\mathcal V}}
\usepackage{mathtools}

\let\mathscr=\mathcal
\usepackage[mathletters]{ucs}
\usepackage[utf8x]{inputenc}
\usepackage[english]{babel}
\expandafter\let\csname ver@ucs.sty\endcsname=\relax

\DeclareUnicodeCharacter{8788}{\coloneqq}
\DeclareUnicodeCharacter{8797}{\coloneqq}
\DeclareUnicodeCharacter{8214}{\|}
\DeclareUnicodeCharacter{10877}{\le}
\DeclareUnicodeCharacter{10878}{\ge}
\DeclareUnicodeCharacter{10971}{\pitchfork}

\makeatletter
\def\@firstofthree#1#2#3{#1}
\expandafter\let\expandafter\t@alpha@prefix\expandafter=\@firstoftwoα
\expandafter\let\expandafter\t@sigma@prefix\expandafter=\@firstoftwoσ
\expandafter\let\expandafter\t@cal@prefix\expandafter=\@firstofthreeℰ

\def\t@greek#1#2{\t@real{#1#2}}
\def\t@cal#1#2#3{\t@real{#1#2#3}}
\def\t@{\@ifnextchar\bgroup{\t@real}{\t@@}}
\def\t@@#1{%
    \let\t@helper=\t@real
    \ifx\t@alpha@prefix#1
        \let\t@helper=\t@greek
    \fi
    \ifx\t@sigma@prefix#1
        \let\t@helper=\t@greek
    \fi
    \ifx\t@cal@prefix#1
        \let\t@helper=\t@cal
    \fi
    \t@helper#1
}

\def\t{\let\t@real=\tilde\t@}
\newcommand{\wt}{\let\t@real=\widetilde\t@}
\makeatother
\fi

\usepackage{csquotes}
\usepackage[backend=biber,firstinits=true,maxnames=4]{biblatex}
\newtheorem{theorem}{Theorem}

\newcommand{\mynewtheorem}[2]{
\newtheorem{#1}{#2}
\expandafter\def\csname #1autorefname\endcsname{#2}
}

\mynewtheorem{proposition}{Proposition}
\mynewtheorem{lemma}{Lemma}
\mynewtheorem{corollary}{Corollary}
\mynewtheorem{conjecture}{Conjecture}
\mynewtheorem{assumption}{Assumption}

\mynewtheorem{problem}{Problem}
\theoremstyle{definition}
\mynewtheorem{definition}{Definition}
\theoremstyle{remark}
\mynewtheorem{remark}{Remark}

\newtheorem{example}{Example}

\newcounter{step}
\newcommand{\step}[1]{\refstepcounter{step}\textbf{Step \thestep.} \emph{#1.}\ }

\DeclareMathOperator{\const}{const}
\DeclareMathOperator{\codim}{codim}

\DeclareMathOperator{\Vect}{Vect}
\DeclareMathOperator{\Sing}{Sing}
\DeclareMathOperator{\Per}{Per}
\DeclareMathOperator{\Sep}{Sep}

\title{Global bifurcations in the two-sphere: a new perspective}

\author{%
Yu. Ilyashenko%
\thanks{National Research University Higher School of Economics, Russia}\ %
\thanks{Cornell University, US}\ %
\thanks{Independent University of Moscow}\ %
\thanks{The authors were supported in part by the grant RFBR 16-01-00748},
\and Yu. Kudryashov\footnotemark[1]\ \footnotemark[2]\ \footnotemark[3],
\and I. Schurov\footnotemark[1]\ \footnotemark[3]\ %
\thanks{Research of I.S. was supported in part by Dynasty Foundation}}

\makeatletter
\hypersetup{pdftitle=\@title,pdfauthor={Yu. Ilyashenko, Yu. Kudryashov, I. Schurov}}
\makeatother

\begin{document}

\maketitle

\begin{abstract}
    We construct an open set of structurally unstable three parameter families whose weak and so called moderate topological classification defined below has a numerical invariant that may take an arbitrary positive value.
    Here and below “families” are “families of vector fields in the two-sphere”.
    This result disproves an Arnold's conjecture of 1985.
    Then we construct an open set of six parameter families whose moderate topological classification has a functional invariant.
    This invariant is an arbitrary germ of a smooth map $(ℝ_+, a)→(ℝ_+, b)$.
    More generally, for any positive integers $d$ and $d'$, we construct an open set of families whose topological classification has a germ of a smooth map $\left(ℝ_+^d, a\right)→\left(ℝ_+^{d'}, b\right)$ as an invariant.
    Any smooth germ of this kind may be realized as such an invariant.
    These results open a new perspective of the global bifurcation theory in the two sphere.
    This perspective is discussed at the end of the paper.
\end{abstract}
\ifprinttoc
\tableofcontents
\fi

\section{Introduction: structurally unstable families}\label{sec:intro}

There are many families of planar vector fields whose bifurcations are investigated up to now.
All of them are weakly structurally stable in their domains (in a neighborhood of a singular point or a polycycle), in a sense explained below.
In 1985 Arnold suggested a perspective of the development of the global bifurcation theory in the two sphere \cite{AAIS}.
In particular, he conjectured that generic families of vector fields considered on the whole sphere
are structurally stable.
The first result of this paper disproves this conjecture.

Arnold includes the conjecture mentioned above in the list of six.
Right after the statement of these conjectures Arnold writes:
\begin{quote}
    \itshape
    “Certainly proofs or counterexamples to the above conjectures are necessary for investigating nonlocal bifurcations in generic $l$-parameter families.”
\end{quote}

The current paper is motivated by the Arnold's conjectures.
As we show, the most nontrivial of them appeared to be wrong.
This opens a new perspective of the development of the global bifurcation theory in the two sphere.
It is discussed in \autoref{sec:perspective}.

\subsection{Basic definitions and notation}\label{sub:basic}
Recall the necessary definitions.
We give them in the general setting, though for our needs everywhere below we may take $M = S^2$.
Here and below $B⊂ℝ^k$ (base of a family) is a topological open ball.

Denote by $\Vect(M)$ the set of $C^3$-smooth vector fields on $M$.
\begin{definition}
    \label{def:fam}
    A \emph{family of vector fields} on a manifold $M$ with the base $B$ is a vector field $V$ on $B×M$ tangent to the fibers $\set{α}×M$, $α∈B$.
    The dimension of a family is the dimension of its base.
\end{definition}
An equivalent definition.
\begin{definition}
    \label{def:fam1}
    A \emph{family of vector fields} on $M$ with the base $B$ is a smooth map $V:B→\Vect(M)$.
\end{definition}
The equivalence is obvious.
Denote by $\mcV_k(M)$ the space of $k$-parameter families of vector fields on $M$ which are $C^3$ smooth as vector fields on $B×M$.
\begin{definition}
    \label{def:trans}
    A family of vector fields is \emph{transversal} to a Banach submanifold $\bfT$ of $\Vect(M)$ provided that the corresponding map $V:B→\Vect(M)$ is transversal to $\bfT$.
\end{definition}

\begin{definition}
    \label{def:oteq}
    Two vector fields $v$ and $\t v$ on a~manifold~$M$ are called \emph{orbitally topologically equivalent},
    if there exists a~homeomorphism~$ M→M$ that links the phase portraits of $v$ and $\t v$,
    that is, sends orbits of~$v$ to orbits of~$\t v$ and preserves their time orientation.
\end{definition}
\begin{definition}
    \label{def:weak-eq}
    Two families of~vector fields $\set{v_{α}| α∈B}$, $\set{\t v_{\t{α}}| \t{α}∈\t B}$ on $M$ are called \emph{weakly topologically equivalent} if there exists a map
    \begin{equation}
        \label{eqn:conj}
        H\colon B×M→\t B×M,\quad H(α, x)=(h(α), H_α(x))
    \end{equation}
    such that $h:B→\t B$ is a homeomorphism, and for each $α ∈ B$ the map $H_α\colon M → M $ is a homeomorphism that links the phase portraits of $v_α$ and $\t v_{h(α)}$.
\end{definition}

Two families are \emph{topologically equivalent} if there exists a \emph{homeomorphism} $H$ with above properties.
Topological classification of families with a very simple dynamics may have functional invariants that occur due to the requirement of the continuity of $H_α$ in $α$ \cite{R}.
Thus, the topological equivalence is too rigid.
On the other hand, weak topological equivalence introduced in \cite{AAIS} is too lousy:
some families with apparently different bifurcations occur to be weakly topologically equivalent.
For this reason we introduce here a new equivalence relation of \emph{moderate topological equivalence}.
This relation is needed and defined in this paper for vector fields having \emph{hyperbolic singular points only}.

For a family $V$ of vector fields, denote by $\Sing V$ the set of all the singular points of the family, by $\Per V$ the union of all the limit cycles, and by $\Sep V$ the union of all the separatrixes of the vector fields of the family.

\begin{definition}
    \label{def:moderate-eq-nonloc}
    We say that $V$ and $\t V$ are \emph{moderately topologically equivalent} provided that there exists a linking map $H$, see~\eqref{eqn:conj}, which is continuous in $(α, x)$ at the set $\overline{\Sing V ∪ \Per V ∪ \Sep V}$.
\end{definition}

Sometimes the parameter value $α=0$ is distinguished in some way.
Then we also consider a local version of the equivalence relation above.
\begin{definition}
    \label{def:moderate-eq-loc}
    Two families $V$ and $\t V$ are \emph{locally moderately topologically equivalent} at $α=0$ provided that there exists a linking map $H$, see~\eqref{eqn:conj}, which is continuous in $(α, x)$ at the set ${\overline{\Sing V ∪ \Per V ∪ \Sep V}∩\set{α=0}}$.
\end{definition}

A definition of moderate topological equivalence for general families of vector fields is given in \cite{GI}, work in progress.
It requires more technical details that we skip here.

\begin{definition}
    We say that a~family of vector fields is \emph{moderately (weakly) structurally stable} if it is moderately (weakly) topologically equivalent to its small perturbations.
\end{definition}

\subsection{Main results}
Our main results are

\begin{theorem}
    \label{thm:unst}
    There exists a non-empty open subset of $\mcV_3(S^2)$ such that each family from this set is moderately structurally unstable.
    Moreover, moderate topological classification of these families has numeric invariant that may take any positive value.
\end{theorem}
\begin{theorem}
    \label{thm:func-unst}
    There exists a non-empty open subset of $\mcV_6(S^2)$ such that the moderate topological classification of families from this set has a functional invariant, namely a germ of a function $ f: (ℝ_+,a) →(ℝ_+,b)$.
    Moreover, any such germ for any positive $a$, $b$ may be realized as an invariant of this classification.
\end{theorem}
More generally,
\begin{theorem}
    \label{thm:fuinv}
    For any positive integers $d$, $d'$ there exists an open subset of $\mcV_k(S^2)$, $k = 3d + 2d' +1$, such that the moderate topological classification of families from this set has a functional invariant, namely a germ of a map $ f: \left( ℝ_+^d, a \right) →\left( ℝ_+^{d'},b \right)$.
    Moreover, any such germ for any positive vectors $a ∈ℝ_+^d$, $b ∈ℝ_+^{d'}$ may be realized as an invariant of this classification.
\end{theorem}

\autoref{thm:func-unst} is a corollary of \autoref{thm:fuinv}, but we state and prove it separately because it is simpler and its proof contains all the main ideas needed for the proof of the more general result.
These theorems hold true if the moderate equivalence in their statement is replaced by the weak equivalence.
The proofs of these modified theorems require more technical details, and we skip them.

\section{Numerical invariants}\label{sec:family}

\subsection{A special class of degenerate vector fields of codimension three} \label{sub:unperturbed}

In this section we prove \autoref{thm:unst} modulo so called \nameref{lem:asym} proved in \autoref{sec:aux}.
We will first describe the degenerate vector fields, then their unfoldings.

\subsubsection{Polycycles in the sphere}

We do not recall here the detailed definitions of polycycles and their monodromy maps;
they may be found in \cite{cent}.
We only mention that a hyperbolic polycycle is a separatrix polygon whose vertexes are hyperbolic saddles, and edges are saddle connections.
The monodromy (or Poincaré map) $Δ_γ$ along a polycycle $γ$, if exists, is defined like the Poincaré map of a cycle.
The only difference is that the map $Δ_γ$ is defined on a half interval, called a semi-transversal, with the vertex on $γ$, rather than on an interval.
If the monodromy map of the polycycle is well defined, then the polycycle is called \emph{monodromic};
if the corresponding semi-transversal points outside (inside) $γ$, then we say that $γ$ is monodromic from the exterior (from the interior).

\begin{remark}
    \label{rem:normalized}
    Interior and exterior domains on the sphere are to be specified.
    We represent the sphere $S^2$ as $ℝ^2∪\set{∞}$.
    We consider vector fields that have an attracting fixed point $A$ surrounded by a repelling hyperbolic limit cycle $γ_∞$ having no other singular points in the domain bounded by $γ_∞$ that contains $A$.
    We suppose that such a point is unique, and place it to infinity.
    To distinguish the limit cycle $γ_∞$, we will call it \emph{limit cycle near infinity}.
    Such a vector field on $S^2$ will be called \emph{normalized}.
    If two normalized vector fields are orbitally topologically equivalent, then the linking homeomorphism brings $∞$ to $∞$.
    For any Jordan curve on $ℝ^2=S^2∖\set{∞}$, the \emph{exterior} domain is the one that contains $∞$.
    For a polycycle in $ℝ^2$ represented as a union of Jordan curves, the \emph{interior} domain is the union of interior domains of these curves.
\end{remark}

In what follows, we consider vector fields close to a fixed normalized vector field.
Clearly, such vector field is orbitally topologically (and even smoothly) equivalent to a normalized vector field.
Therefore, we can normalize a family so that all vector fields of the normalized family are normalized.
We shall always assume that our families are normalized.

\subsubsection{A special polycycle of codimension three} \label{sub:spec1}

Consider a (normalized) vector field $v$ which has a polycycle $γ$ with two vertexes and three edges, see \autoref{fig:unperturbed}.
The vertexes are hyperbolic saddles $L$ and $M$ with the characteristic numbers $λ$ and $μ$.
Recall that the \emph{characteristic number} of~a~saddle is the modulus of the ratio of its eigenvalues, the negative one in the numerator.
Suppose that
\begin{align}
    \label{eqn:char}
    λ &< 1, &λ^2μ &> 1.
\end{align}

The edges are: a time oriented separatrix loop $l$ of $L$, and two time oriented saddle connections: $LM$ and $ML$, see \autoref{fig:unperturbed}.
These two connections form a polycycle “heart”:
two other separatrixes of $M$ are inside, and those of $L$ are outside this polycyle.
The polycycle $γ$ is monodromic from the exterior, and the loop $l$ from the interior.
The polycycle “heart” is not monodromic at all.

A polycycle $γ$ that satisfies these assumptions is called \emph{a polycycle of type $TH$} ($H$ of “heart”, and $T$ of “tear” that resembles the separatrix loop $l$).

\subsubsection{Vector fields of class \texorpdfstring{$\bfT$}{T}}\label{subsub:class-bfT}

Suppose that the vector field $v$ described above has a saddle $E$ outside the polycycle $γ$ and a~saddle~$I$ inside the separatrix loop~$l$ of~$L$.
The letters $E$ and $I$ come from the words “\emph{exterior}” and “\emph{interior}”.
Suppose that one of the unstable separatrixes of $E$ winds onto $γ$, and one of the stable separatrixes of~$I$ winds onto $l$ in the negative time.
Denote these separatrixes by $W^u_E$ and $W^s_I$.
In \autoref{sec:aux} we prove that inequalities~\eqref{eqn:char} imply the possibility of this winding, see \autoref{rem:wind}.
Polycycles described above may occur in generic three-parameter families.
Existence of separatrixes $W^u_E$ and $W^s_I$ winding to $γ$ and from $l$ does not increase the codimension of the degeneracy.

Suppose that there exists a smooth oriented arc $Γ$ that goes from $I$ to $E$, intersects $γ$ at a unique point $O$ and is transversal to $v$ strictly between its ends.
Denote by $Γ^+$, $Γ^-$ half open arcs of $Γ$ between $O$ and $E$ ($E$ excluded), between $O$ and $I$ ($I$ excluded) respectively.
The germs of the monodromy maps $Δ_γ:(Γ^+,O)→(Γ^+,O)$, $Δ_l:(Γ^-,O)→(Γ^-,O)$ along the polycycle $γ$ and the loop $l$ are well defined.
Suppose that the germs $Δ_γ, Δ_l^{-1}$ may be extended to the monodromy maps
\begin{align}
    \label{eqn:mon}
    Δ_γ&: Γ^+→Γ^+, &Δ_l^{-1}&: Γ^- → Γ^-
\end{align}
that have no fixed points except for $O$.

\begin{assumption}
    \label{as:*}
    The vector field $v$ has exactly two saddles, namely $E$ and $I$, with the following properties:
    a separatrix of the first saddle winds towards the polycycle $γ$;
    a separatrix of the second saddle winds towards the loop $l$ in the negative time.
    That is, no other saddle has one of these properties.
    Moreover, all the singular points of the vector field $v$ are hyperbolic.
\end{assumption}

The set of vector fields with these properties is called class $\bfT$.
\begin{figure}
    \centering\includegraphics{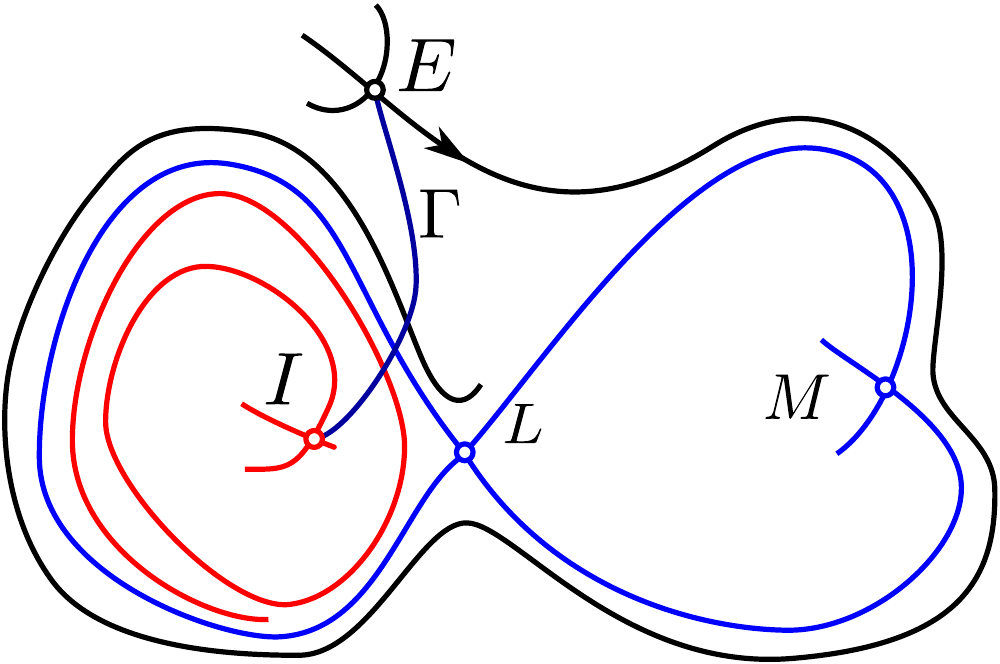}
    \caption{The phase portrait of the degenerate vector field $v$}
    \label{fig:unperturbed}
\end{figure}

Denote by $\mathring\bfT$ the set of vector fields $v\in\bfT$ that have no saddle connections, except for the edges of the polycycle $γ$.
In particular, a vector field $v∈\mathring\bfT$ has a unique polycycle homeomorphic to $γ$.
As we shall show in \autoref{subsub:ban}, the set $\mathring\bfT$ is a Banach submanifold of $\Vect(S^2)$ of codimension $3$.
Let $\bfC$ be the class of $3$-parameter families $V$ such that
\begin{itemize}
    \item $V$ has a unique intersection $v$ with $\bfT$;
    \item $v∈\mathring\bfT$;
    \item $V$ is transversal to $\mathring\bfT$ at $v$, see \autoref{def:trans}.
\end{itemize}

Denote by $\mcH$ (of \emph{hyperbolic}) the class of vector fields $v\in\Vect(S^2)$ such that
\begin{itemize}
    \item all singular points of $v$ are hyperbolic;
    \item the characteristic number of any saddle of $v$ is different from $1$, and from the inverse square of the characteristic number of any other saddle of $v$.
\end{itemize}
Obviously, $\mcH$ is open and dense in $\Vect(S^2)$, and $\mcH∩\bfT$ is dense in $\bfT$.

Consider a family $V∈\bfC$ intersecting $\bfT$ at $v∈\mathring\bfT$.
Let $γ$ be the polycycle of $v$ of the type $TH$;
let $λ$, $μ$ be the characteristic numbers of the saddles $L$, $M$, respectively.
Let us define
\begin{equation}
    \label{eqn:nu-def}
    ν(V)=\frac{-\log λ}{\log λμ^2}.
\end{equation}

The following theorem is a refinement of \autoref{thm:unst}.
\begin{theorem}
    \label{thm:unst1}
    The number $ν$ given by \eqref{eqn:nu-def} is a numeric invariant of the moderate classification of families $V∈\bfC$ such that $V⊂\mcH$,
    i.e. for two moderately equivalent families $V, \t V∈\bfC$ such that $V, \t V⊂\mcH$, we have
    \begin{equation}
        \label{eqn:unst1}
        ν(V)=ν(\t V).
    \end{equation}
\end{theorem}
The ratio $ν$ given by \eqref{eqn:nu-def} may be expressed through $\frac{\log λ}{\log μ}$.
In view of~\eqref{eqn:char}, the latter ratio takes the values in $\left(-\frac 12, 0\right)$.
The ratio $ν$ may take arbitrary positive values.

\autoref{thm:unst} follows from \autoref{thm:unst1}, so it is enough to prove the latter theorem.

We introduced the class $\mcH$ because $\bfT$ is topologically distinguished in this set.
\begin{proposition}
    \label{prop:bfT-dist}
    Consider two orbitally topologically equivalent vector fields $v, \t v∈\mcH$.
    If $v∈\bfT$, then $\t v∈\bfT$.
    Moreover, a homeomorphism provided by \autoref{def:oteq} sends the polycycle $γ$, and the saddles $L$, $M$, $E$, $I$ of the vector field $v$ to similar objects for $\t v$.
\end{proposition}
\begin{proof}
    Let $H_0$ be a homeomorphism provided by \autoref{def:oteq}.
    Let $\t L$, $\t M$, $\t E$, $\t I$, $\t l$, $\t γ$, $\t W^s_I$, $\t W^u_E$, $\t Γ$ be the images of the corresponding objects for $v$ under $H_0$.
    Since $\t v∈\mcH$, the singular points $\t L$, $\t M$, $\t E$, $\t I$ are hyperbolic saddles for $\t v$.
    Next, $\t γ$ is a polycycle homeomorphic to $γ$, and $\t W^s_I$ (resp., $\t W^u_E$) is a separatrix of $\t I$ (resp., $\t E$) winding onto $l$ (resp., $γ$) in negative (resp., positive) time.
    Obviously, $\t E$, $\t I$ are unique saddles whose separatrixes wind onto $γ$, $l$.
    Indeed, if $\t W'^u_E$ is another separatrix winding onto $γ$, then its preimage under $H_0$ winds onto $γ$, which contradicts \autoref{as:*}.

    Finally, $\t γ$ satisfies all requirements of class $\bfT$, possibly except for \eqref{eqn:char}.
    Let $\t λ$, $\t μ$ be the characteristic numbers of $\t L$, $\t M$.
    Suppose that $\t λ≮1$.
    By definition of $\mcH$, we have $\t λ≠1$, $\t λ^2\t μ≠1$, hence $\t λ>1$.
    Due to \autoref{rem:wind} below, this implies that $\t W^s_I$ cannot wind onto $\t l$ in the negative time.
    Analogously, $\t λ^2\t μ≯1$ implies $\t λ^2\t μ<1$, hence $\t W^u_E$ cannot wind onto $\t γ$.
    Finally, $\t v∈\bfT$.
\end{proof}

\subsubsection{Extension to the whole sphere}
\label{subsub:ext}

Note that we described above only a part of the phase portrait of a vector field $v$ of class~$\bfT$.
Such a vector field may be extended to the whole sphere in many different ways.
One extension of a slightly perturbed vector field is shown in \autoref{fig:conn}.
Let us describe an extension of $v$, beginning with the interior of the loop $l$.
There are exactly three singular points inside this loop: the saddle $I$, one sink and one source.
They are hyperbolic by \autoref{as:*}.
The outgoing separatrixes of the saddle $I$ tend to the sink.
The closure of their union bounds a domain homeomorphic to a disc.
There is one source in this domain; the interior of this domain is the repulsion basin of this source.
The closure of this basin is a so called Cherry cell, see \autoref{fig:conn}.

The phase portrait of $v$ outside the polycycle $γ$ is constructed in a similar way.
The polycycle $γ$ is surrounded by a repelling hyperbolic limit cycle;
two incoming separatrixes of $E$ tend to this cycle in the negative time;
the domain bounded by these separatrixes is a Cherry cell, see \autoref{fig:conn}.
In the simplest case, this limit cycle is the limit cycle near infinity.
In other cases, the phase portrait outside this limit cycle may be more complicated;
in particular, it may include other polycycles homeomorphic to $γ$.

The interior of the “heart” part of the polycycle $γ$ (union of connections $LM$ and $ML$) contains one sink, one source and no limit cycles.
One outgoing separatrix of the saddle $M$ tends to the sink.
One incoming separatrix of $M$ emerges from the source.

The extension described above is in a sense a key example.
In what follows, we will describe only a part of the phase portraits of the vector fields considered, paying no attention to the extension of these parts to the whole sphere.
These extensions are analogous to the example described above.

\subsection{Generic unfoldings of vector fields of class \texorpdfstring{$\bfT$}{T}}

\subsubsection{Local families and unfoldings}\label{subsub:unfoldings}

In what follows, we deal with local families.

\begin{definition}
    \label{def:local}
    A \emph{local family} at $α = 0$ with the base $(ℝ^k, 0)$ is a germ on $\set{0} × M$ of a family given on $B × M$, $B ∋ 0$, $B$ is open.
    Two local families are \emph{moderately topologically equivalent} if they have locally moderately topologically equivalent representatives, and the corresponding homeomorphism of the bases maps $0$ to $0$.
\end{definition}

Denote by $\mcV_k^{loc}(S^2)$ the set of local $k$-parameter families of $C^3$ smooth vector fields in $S^2$.

\begin{definition}
    \label{def:unfol}
    An \emph{unfolding} of a vector field is a local family for which this field corresponds to the critical (zero) parameter value.
    We say that this family \emph{unfolds} the vector field.
\end{definition}

We will be mainly interested in generic unfoldings.

For a Banach submanifold $\bfT⊂\Vect(S^2)$, denote by $\bfT^\trans_k$ the set of $k$-parameter unfoldings of vector fields $v∈\bfT$ transversal to $\bfT$ at $v$.

\subsubsection{Genericity assumptions}\label{sub:gencond}

Consider an unperturbed vector field $v ∈ \bfT$.
Let $γ$ be the polycycle of the vector field $v$ of the type $TH$.
If $v$ has several polycycles homeomorphic to $γ$, we fix any of them, say $γ_j$, and let $γ = γ_j$.

Fix cross-sections $Γ_2$, $Γ_3$ to the arcs $LM$ and $ML$ of $γ$, respectively, and let $Γ_1 = Γ$, see \autoref{fig:conn}.
Let $O_j$ be the only intersection point of $γ$ with $Γ_j$, $O_1 = O$.
Fix coordinates $x_j\colon Γ_j→ℝ$ that orient $Γ_j$ from inside to outside of $γ$.

\begin{figure}
    \centering\includegraphics[scale=0.6]{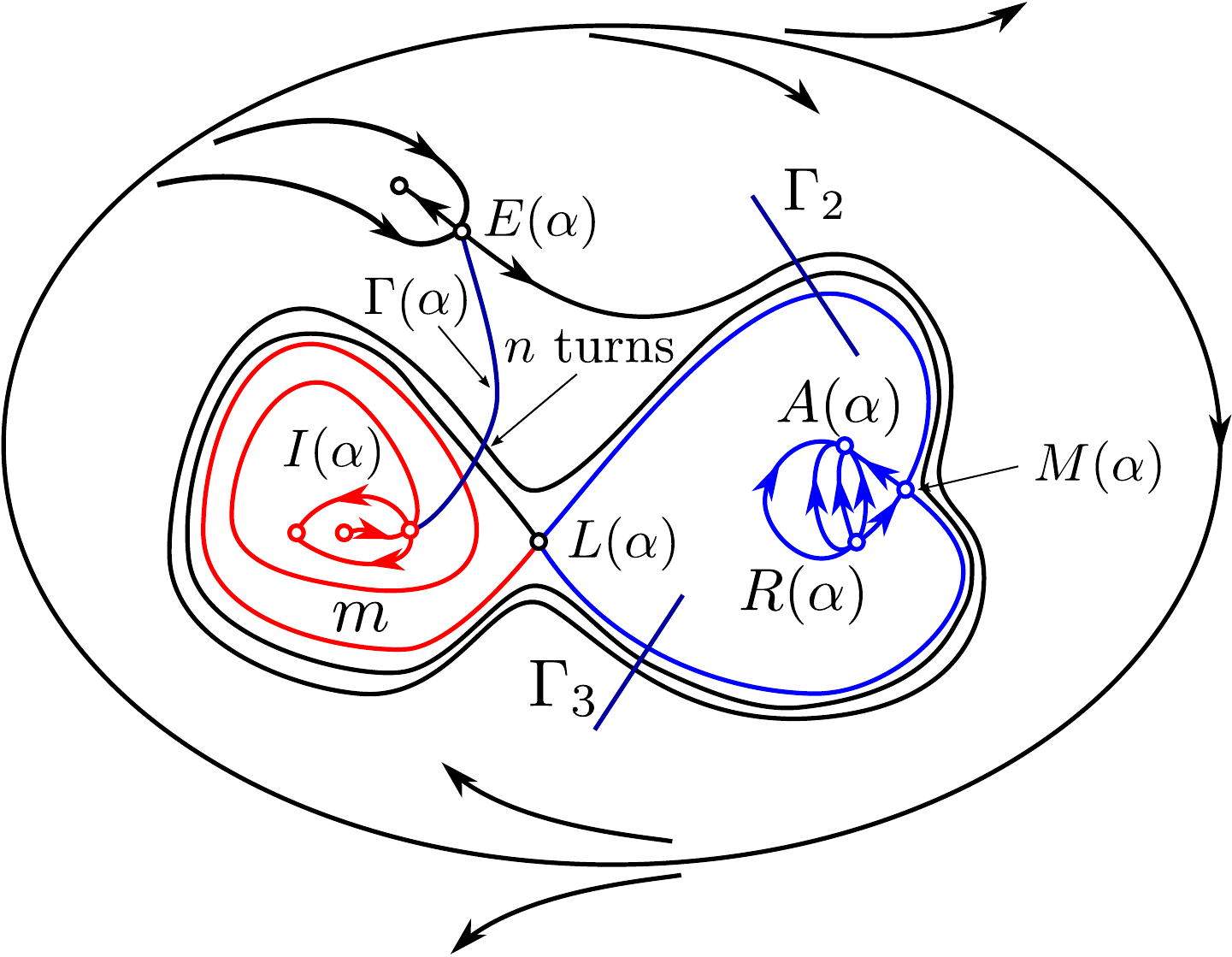}
    \caption{Sparkling saddle connections of an unfolding of $v$ and an extension of the phase portrait to the whole sphere}
    \label{fig:conn}
\end{figure}

Now consider a three parameter local family $V = \set{ v_α }$ that unfolds $v$.
Let $α∈(ℝ^3, 0)$ be the parameter, and $v_α$ be the corresponding vector field of the family, $v= v_0$.
For $α$ close to zero, the vector field $v_α$ has two saddles $L(α)$ and $M(α)$ smoothly depending on $α$, together with their separatrixes, $L(0) = L$, $M(0) = M$.

For the initial vector field, each point $O_j$ belongs to two separatrixes, one unstable and one stable.
For $α$ close to zero, let $U_j(α)$ and $S_j(α)$ be the first intersection points of the separatrixes of the saddles $L(α)$ and $M(α)$ with the cross-sections $Γ_j$;
$U$ corresponds to the unstable separatrixes, $S$ corresponds to the stable ones;
$S_j(0) = U_j(0) = O_j$.

Let us introduce \emph{separatrix splitting parameters} $σ_j$ well defined for small $α$:
\begin{equation}
    \label{eqn:sep-split}
    σ_j(α) = x_j(S_j(α))-x_j(U_j(α)).
\end{equation}
We also introduce the \emph{separatrix splitting map},
\begin{equation}
    \label{eqn:sepmap}
    σ(α)=(σ_1(α), σ_2(α), σ_3(α)).
\end{equation}
The only genericity assumption on the unfolding of $v$ is
\begin{equation}
    \label{eqn:deg-cond}
    \det \frac{∂σ}{∂α}(0)\ne 0.
\end{equation}
It is easy to check that this condition does not depend on the choice of $Γ_j$ and $x_j$.
Relation \eqref{eqn:deg-cond} allows us to take $σ$ as a new parameter of the local family $V$.
\begin{remark}
    Though $Γ_1=Γ$ depends on $α$ (see \autoref{subsub:detailed} below), we may and will assume that the intersection $Γ_1'$ of $Γ_1$ with a small neighborhood of $O$ does not depend on $α$, hence we can choose $x_1\colon Γ_1'→ℝ$, and define $σ_1(α)$ for $α$ small enough.
    Local cross sections $Γ_2$ and $Γ_3$ do not depend on $α$;
    hence, $σ_2(α)$ and $ σ_3(α)$ are well defined.
\end{remark}

\subsubsection{Class \texorpdfstring{$\bfT$}{T} and Banach submanifolds}
\label{subsub:ban}
Let us prove that the class $\bfT$ is a Banach submanifold in some neighborhood of any point $v∈\mathring\bfT$.
Such a vector field is a generic point in $\bfT$.
By definition of a Banach submanifold, it is sufficient to find a map $σ:(\Vect(S^2), v)→(ℝ^3, 0)$ such that $(\bfT, v)=σ^{-1}(0)$, and rank of $dσ(v)$ equals $3$.
The map $σ$ is given by \eqref{eqn:sepmap}.
Since $v$ has no saddle connections except for the edges of $γ$, and all singular points of $v$ are hyperbolic, the equation $σ=0$ actually defines $\bfT$ near $v$.
Next, the map $σ$ has a full rank since we may use a partition of unity to perturb one of $σ_j$ without touching the other two.
Note that relation \eqref{eqn:deg-cond} is equivalent to $V∈\bfT^\trans_3$.

Near the points $v$ corresponding to vector fields that have several polycycles homeomorphic to $γ$, say $γ_1$, \dots, $γ_D$, the set $\bfT$ is no more a manifold.
Let $v$ be such a field.
Denote by $\bfT_{γ_1}$ the set of those vector fields $w∈\bfT$ close to $v$ that have a polycycle $γ_1(w)$ homeomorphic to $γ$, continuously depending on $w$ and such that $γ_1(v)=γ_1$.
The set $\bfT_{γ_1}$ is a Banach manifold in some neighborhood of $v$.
The set $\bfT$ near $v$ is a union of the manifolds $\bfT_{γ_1}$, \dots, $\bfT_{γ_D}$, and has self-intersections.

\subsubsection{Sparkling saddle connections in a family of class \texorpdfstring{$\bfT^{\trans}_3$}{T\textasciicircum ⫛\_3}: a heuristic description}
\label{sub:sparkling}

Let a vector field $v∈\bfT$ have exactly one polycycle $γ$ of the type $TH$.
Let $V$ be a three-parameter unfolding of $v$ transversal to $\bfT$;
$σ = (σ_1, σ_2, σ_3)$ be the parameter of this family defined above.

Consider the line $ℰ = \set{σ_2 = σ_3 = 0}$ in the parameter space.
It is parametrized by $σ_1 $, and corresponds to vector fields $v_α$ with the two connections between $L(α)$ and $M(α)$ unbroken.
The parameter $σ_1$ is special; let us redenote it by $ε$.

Let us give a heuristic description of some bifurcations in the family $V$ that occur as the parameter~$α$ changes along $ℰ$.
Let $I(α), E(α)$ be hyperbolic saddles of $v_{α}$ smoothly depending on $α$, $I(0) = I$, $E(0) = E$.
There are two sequences of sparkling saddle connections:
\begin{description}
    \item[exterior connection:]
        the outgoing separatrix $W^u_E(α)$ of the saddle $E(α)$ makes $n$ turns (the precise meaning of this term is defined below) around the whole polycycle $γ$,
        then coincides with the incoming separatrix of the saddle $L(α)$;

    \item[interior connection:]
        the incoming separatrix $W^s_I(α)$ of the saddle $I(α)$ makes $m$ turns along $l$ in the backward time,
        then coincides with the outgoing separatrix of the saddle $L(α)$.
\end{description}

These connections (that do not in general occur simultaneously) are shown in~\autoref{fig:conn}.
For simplicity, they are shown in one and the same figure.

\subsubsection{Sparkling saddle connections in a family of class \texorpdfstring{$\bfT^{\trans}_3$}{T\textasciicircum ⫛₃}: a detailed description}\label{subsub:detailed}

Let us describe in more detail the bifurcations in the family $V$ that occur when the parameter changes along the line $ℰ$.
Let $Γ(α)$ be a smooth family of arcs connecting $I(α)$ with $ E(α)$ and transversal to $v_{α}$, $Γ (0) = Γ $.
\begin{definition}
    \label{def:n-turns-numer}
    We say that an exterior (resp., an interior) saddle connection $ E(α) L (α)$ (resp., $L(α)I(α)$) \emph{makes $n$ turns} around $γ$ (resp., $m$ turns around $l$) provided that it intersects the curve $ Γ (α)$ with endpoints excluded at exactly $n$ points (resp., $m$ points).
\end{definition}
\begin{remark}
    The number of turns \emph{is not} an invariant of the orbital topological classification of vector fields on the two-sphere.
    Indeed, $n$ and $m$ depend not only on $v_α$, but also on the choice of $Γ(α)$.
\end{remark}

We fix the family $Γ(α)$, and thus speak about the number of turns.

The following lemma describes the asymptotic behavior of~the corresponding values of the parameter $ε$ near zero.
\begin{lemma}
    [Asymptotic Lemma]
    \label{lem:asym}
    For $m$ and $n$ sufficiently large the following holds.
    There exists only one value $ε=i_m$ (respectively, $ε=e_n$) that corresponds to a vector field with an interior (respectively, an exterior) sparkling separatrix connection with $m$ (respectively, $n$) turns.
    The sequences $i_m$ and $e_n$ satisfy the asymptotic equalities
    \begin{align}
        \label{eqn:im}
        \log (-\log i_m) &= m\log λ^{-1} + O(1),\\
        \label{eqn:en}
        \log (-\log e_n) &= n\log(λ^2μ)+ O(1),
    \end{align}
    and are monotonically decreasing.
\end{lemma}

\begin{remark}
    The choice of the cross section $Γ$ and the family $Γ(α)$ does not affect the statement of the lemma:
    another choice may change the number of turns by a constant term only.
\end{remark}

\begin{remark}
    In fact, for a sufficiently smooth family of~vector fields, $O(1)$ in~\eqref{eqn:im} and~\eqref{eqn:en} can be replaced by $a+r_m$ and $b+\rho_n$, respectively, where $a$ and $b$ are constants and $r_m$, and $\rho_n$ tend to $0$ exponentially.
    We do not need this fact for our analysis.
    Its proof requires more technical details, so we will skip it.
\end{remark}

\autoref{lem:asym} is the most technical part of the proof.
We postpone its proof to \autoref{sec:aux}.
Let us deduce the theorem from it.

\subsection{Two equivalent families of class \texorpdfstring{$\bfT^\trans_3$}{T\textasciicircum ⫛₃}}\label{sub:twoequiv}
Begin with some notes about the real line.

\subsubsection{Topology of two sequences in the line}
\label{subsub:relative-density}
Parameter values $e_n$ corresponding to exterior sparkling saddle connections and $i_m$ corresponding to interior ones form two decreasing sequences converging to $0$ from above.
As shown below, for two equivalent families of class $\bfT^\trans_3$ there exists a homeomorphism $(ℝ_+,0) →(ℝ_+,0)'$ of neighborhoods of zero on the real line $ℰ$ parametrized by $ε$ that brings the couple of sequences $(i_m)$, $(e_n)$ corresponding to one family to the couple of analogous sequences corresponding to another one.

\begin{definition}
    \label{def:coup-eq}
    Two couples of bounded countable sets $A$, $B$ and $\t A$, $\t B$ in $ℝ_+$ with the only accumulation point zero are called \emph{equivalent} provided that there exists a germ of a homeomorphism $h: (ℝ_+,0) →(ℝ_+,0)$ that links the germs of the corresponding sets at zero:
    \begin{align*}
        h(A,0) &= (\t A,0), &h(B,0) &= (\t B,0).
    \end{align*}
\end{definition}

Equivalence classes of pairs thus defined may have a lot of topological invariants.
We present here only one called \emph{relative density}.
Enumerate the elements of each of the countable sets in the definition above in a monotonic decreasing order.
Let $N_A(x)$ be the \emph{counting function}
\[
    N_A(x) = \# \set{a∈A|a≥x}.
\]
In the same way the counting function is defined for the other sets $B$, $\t A$, $\t B$.
The \emph{relative density} of two sets $A$ and $B$ is the limit (if exists):
\[
    ν(A,B) = \lim_{x→0} \frac {N_A(x)}{N_B(x)}.
\]
Clearly, $ν(A, B)$ depends only on the germs of $A$ and $B$ at zero, and $ν(A, B)=ν(\t A, \t B)$ provided that $(A, B)$ is equivalent to $(\t A, \t B)$.

\autoref{lem:asym} motivates the following definition
\begin{definition}
    A set $A⊂ℝ$ is called a \emph{quasi (arithmetic) progression} with the difference $δ_A$ if its germ at infinity has the form
    \[
        (A, +∞)=\set{δ_Ak+O(1)|k∈ℕ}.
    \]
\end{definition}

\begin{lemma}
    \label{lem:dens}
    Let $A, B⊂ℝ_+$ be two countable bounded sets with the only accumulation point at zero.
    Suppose that there exists a homeomorphism $ℒ\colon (ℝ_+, 0)→(ℝ_+, +∞)$ such that $ℒ(A)$ and $ℒ(B)$ are quasi progressions with differences $δ_A$ and $δ_B$, respectively.
    Then the relative density $ν(A, B)$ exists and is given by
    \[
        ν(A, B) = \frac {δ_B}{δ_A}.
    \]
\end{lemma}

The proof is obvious.
This lemma together with \autoref{lem:asym} immediately implies the following statement.
\begin{corollary}
    \label{cor:spar}
    The sequences $(e_n)$, $(i_m)$ introduced in \autoref{lem:asym} have the relative density
    \begin{equation}
        \label{eqn:Ni-over-Ne}
        ν\left( (e_n), (i_m) \right)=\frac{-\log λ}{\log λ^2μ}.
    \end{equation}
\end{corollary}

\subsubsection{A topologically distinguished one-parameter subfamily}

Let $V$ and $ℰ$ be the same as in \autoref{sub:sparkling}.
We will describe $ℰ$ in geometric terms.

\begin{proposition}
    \label{prop:dist}
    Let $L$ and $M$ be two hyperbolic saddles of a $C^1$-smooth vector field $v_0$ on $S^2$.
    Suppose that $v_0$ has a saddle connection $l$ between $L$ and $M$.
    Let $σ_l$ be a separatrix splitting parameter $v ↦ σ_l(v)$ defined for the vector fields $C^1$-close to $v_0$ as it was done in \autoref{sub:gencond}, see \eqref{eqn:sep-split}.
    Then there exists a neighborhood $U$ of $l$ in $S^2$and a neighborhood $W$ of $v_0$ in $\Vect(S^2)$ such that if a vector field $v ∈ U$ has a saddle connection that belongs to $U$ then $σ_l(v) = 0$.
    The proposition holds true if we replace $\Vect(S^2)$ by a family that unfolds $v_0$, and $W$ is a neighborhood of 0 in the base of the family.
\end{proposition}

The proof is trivial.
Yet the proposition becomes wrong if we omit the requirement that the connection between the saddles $L(v)$ and $M(v)$ close to $L$ and $M$ is close to the original connection itself.
This happens due to sparkling saddle connections.

By \autoref{prop:dist}, the family $ℰ$ is characterized by the following condition:
for $σ$ small, the vector field $v_σ$ has a polycycle “heart” with the edges close to the saddle connections $LM$ and $ML$ of the vector field $v_0$ if and only if $σ_2 = σ_3 = 0$, that is, $σ ∈ ℰ$.

\subsubsection{Numeric invariants of families of class \texorpdfstring{$\bfT^{\trans}_3$}{T\textasciicircum ⫛₃}} \label{subsub:num}
Let $V$ and $\t V$ be two families as in \autoref{thm:unst1}.
Let $H = (h, H_{α})$ be a moderate topological equivalence~\eqref{eqn:conj} that links the families $V$ and $\t V$, $h$ be the corresponding map of the parameter spaces.
Let us shift coordinates in the bases of $V$ and $\t V$ so that $v_0$, $\t v_0$ are unique points of $V∩\bfT$ and $\t V∩\bfT$, respectively.
\autoref{prop:bfT-dist} implies that $h(0)=0$.

In this subsection, redenote by $V$, $\t V$ the \emph{local} families $(V, v)$, $(\t V, \t v)$ parametrized by $σ$ and $\t σ$.
Let $\t ℰ$ be the local subfamily of $\t V$, analogouse to $ℰ⊂V$, i.e. given by equations $\t σ_2=\t σ_3=0$.
We want to prove that $h(ℰ)=\t ℰ$.

Let $γ$ be the polycycle of $v_0$ of the type $TH$, $L$ and $M$ be its vertexes.
Due to \autoref{prop:bfT-dist}, $\t γ = H_0(γ)$ is a hyperbolic polycycle of $\t v_0$ with vertexes $\t L=H_0(L)$, $\t M=H_0(M)$.
Let $U$ and $W$ be the neighborhoods of $\t L\t M$ in $S^2$ and $0$ in $ℝ^3$ from \autoref{prop:dist}.
If $σ_2 = 0$, then the vector field $v_σ$ has a saddle connection $L(σ) M(σ)$ close to $LM$.
As $H$ is moderate, the saddle connections $\t L(σ) \t M(σ)=H_σ(L(σ)M(σ))$ belong to $U$ for $σ$ small, and $h(σ)$ belongs to $W$.
By \autoref{prop:dist}, $\t σ_2 (h (σ)) = 0$.
Same arguments imply that if $σ_3 (σ)= 0$ then $ \t σ_3 (h(σ) ) = 0$ for $σ$ small.
Hence, $h (ℰ) = \t ℰ$.

Let $ε = σ_1$, $\t ε = \t σ_1$ be the charts on $ℰ$, $\t ℰ$ respectively, and $ψ:(ℝ, 0)→(ℝ, 0)$ be the restriction $h|_ℰ$ written in the charts $ε$, $\t ε$.
It is a homeomorphism.
Let $\t E = H_0 (E)$.
By \autoref{prop:bfT-dist}, $\t E$ is the unique saddle whose unstable separatrix winds towards the polycycle $\t γ$.
Let $(\t e_n) $ be the sequence of the parameter values $\t ε$ for which the vector field $\t v_σ$ has a sparkling saddle connection between $\t E(\t σ)$ and $\t L(\t σ)$ for $\t σ = (\t ε, 0,0)$.
The phase portraits of the vector fields $v_σ$ and $\t v_{\t σ}$, $\t σ=h(σ)$, are linked by a homeomorphism.
Hence, they simultaneously have or have not sparkling saddle connections between $E(σ)$, $L(σ)$ and $\t E(\t σ)$, $\t L(\t σ)$.
Thus $ψ$ sends the sequence $(e_n)$ to the sequence $(\t e_n)$, possibly shifting the numeration.
In the same way, the sequence $(\t i_m)$ is defined, and the relation $ψ (i_m) = \t i_{m-a}$ for some $a$ is proved.

Due to \autoref{cor:spar}, the relative densities $ν\left( (e_n), (i_m) \right)$ and $ν\left( (\t e_n), (\t i_m) \right)$ exist and are given by \eqref{eqn:Ni-over-Ne}.
Since the pairs of sequences $\left( (e_n), (i_m) \right)$ and $\left( (\t e_n), (\t i_m) \right)$ are equivalent in sense of \autoref{def:coup-eq}, their relative densities coincide.
Hence,
\[
    \frac{\log λ^{-1}}{\log λ^2μ}
    =
    \frac{\log \t λ^{-1}}{\log\t λ^2\t μ}.
\]
\autoref{thm:unst1}, modulo \autoref{lem:asym}, is proved.
\section{Functional invariants}\label{sec:finv}

In this section we construct an open set of families having functional invariants;
in particular, we prove Theorems \ref{thm:func-unst} and~\ref{thm:fuinv}, modulo \autoref{lem:asym1} stated below.
This lemma is an analog of \autoref{lem:asym}.
It is proved in \autoref{sec:aux}.

\subsection{Heuristic description of functional invariants}
\label{sub:heuristic-func}

Consider the class $\bfT$ described in \autoref{sub:unperturbed}.
Choose a vector field $v$ of this class, and a polycycle $γ$ of $v$ of the class $TH$;
if it is not unique, we fix one arbitrary.
For vector fields $w$ close to $v$ on $\bfT_γ$, the following objects are defined:
\begin{itemize}
    \item two saddles $L(w)$ and $M(w)$;
    \item their characteristic numbers $λ (w)$ and $μ (w)$;
    \item the function $ν_γ:(\bfT_γ, v)→(ℝ_+, ν_γ(v))$ given by
    \[
        ν_γ(w) = \frac {-\log λ(w)}{\log λ^2(w)μ (w)},
    \]
    cf. \eqref{eqn:unst1}.
\end{itemize}

Class $\bfT_γ$ has codimension $3$.
Consider the class $\left( \bfT_γ \right)^\trans_4$, see \autoref{subsub:unfoldings}.
A family $V$ of class $\left( \bfT_γ \right)^\trans_4$ intersects the class $\bfT_γ$ at a one-parameter family $J_V$.
For a generic family $V∈\left( \bfT_γ \right)^\trans_4$, the function $ν|_{J_V}$ is a \emph{natural parameter} on $J_V$.

Consider now a class $\bfT^2⊂\bfT$ of vector fields that contain two polycycles $γ_1$, $γ_2$ of the type $TH$.
Let $\bfT_{γ_1} ⊂ \bfT$ and $\bfT_{γ_2} ⊂ \bfT$ be Banach manifolds described in \autoref{subsub:ban}.
Then $\bfT^2 = \bfT_{γ_1} \cap \bfT_{γ_2} ⊂ \bfT$.

As it was mentioned above, the objects $L_j(v), M_j(v), λ_j(v), μ_j(v)$ and $ν_j(v)$ are well defined on the manifold $\bfT_{γ_j}$, $j = 1, 2$.
The class $\bfT^2$ has codimension six.
There are two functions $ν_j=ν_{γ_j}$ defined on the manifolds $\bfT_{γ_j}$.

Consider the class $\left(\bfT^2\right)^\trans_7$.
A family $V∈\left(\bfT^2\right)^\trans_7$ again intersects the class $\bfT^2$ by a one-parameter family $J_V$.
Now, there are two natural parameters on $J_V$:
restrictions to $J_V$ of the functions $ν_1$ and $ν_2$.
One natural parameter is a function of another one.
This is a functional invariant of the moderate topological classification of families of class $\left(\bfT^2\right)^\trans_7$.

In a similar way, increasing the order of degeneracy, one can construct classes of vector fields with $D$ functions like $ν(v)$, and consider families that intersect this class by $d$-dimensional subfamilies, thus obtaining functional invariant of the form of a germ of a map $(ℝ^d_+,a)→(ℝ^{d'}_+, b)$, $d' = D - d$, $a∈ℝ_+^d$, $b∈ℝ_+^{d'}$.

The latter conclusion needs to be justified.
We pass to the rigorous presentation and start with a notion of a function invariant under the moderate topological classification of local families.
In what follows, we construct six (not seven) parameter families with functional invariants;
the constructions are a little bit more tricky than above.
Notation $\bfT^2$ will not be used any more.

\subsection{Families with invariant functions: definition}
\label{sub:inv-functions}

Let $d$, $d'$ be arbitrary natural numbers; later on they will be the same as in \autoref{thm:fuinv}.
Put $D=d+d'$.
As in \autoref{sec:family}, we shall construct a Banach submanifold $\bfT_D$, then consider the class $\left(\bfT_D\right)^\trans_k$ of local $k$-parameter families transversal to this submanifold.
The main difference with \autoref{sec:family} is that now the codimension of $\bfT_D$ is less than the number of parameters in the family,
\[
    \codim \bfT_D=2D+1<2D+1+d=k.
\]
Therefore, each family $V$ with $k>\codim \bfT_D$ includes a subfamily $J_V=V∩\bfT_D$ with ${k-\codim \bfT_D=d}$ parameters.

\begin{definition}
    \label{def:inv}
    Let $T$ be a Banach submanifold in some open subset of $\Vect(S^2)$, $k$ be an integer number, $k≥\codim T$.
    A function $φ:T→ℝ$ is an \emph{invariant function} of the moderate topological classification of families of class $T^\trans_k$ provided that the following holds.
    Let two families $V, \t V∈T^\trans_k$ with the bases $B$, $\t B$ be moderately topologically equivalent.
    Consider the subfamilies $J_V=V∩T$, $J_{\t V}=\t V∩T$ with the bases $B(J_V)⊂B$, $B(J_{\t V})⊂\t B$, respectively.
    Then
    \begin{align}
        \label{eqn:invfun}
        h(B(J_V))&=B(J_{\t V}),&
        φ(v_α)&=φ(\t v_{h(α)})&\text{for }v_α∈J_V.
    \end{align}
\end{definition}
The first relation means that $T$ is topologically distinguished in its neighborhood.
The second relation means that $φ$ takes the same values on the corresponding vector fields of the equivalent families.

Denote by $φ_V:B(J_V)→ℝ$ the function $φ|_{J_V}$ written in the chart $α$, $φ_V(α)=φ(v_α)$;
denote by $χ$ the restriction $h|_{J_V}$ written in the charts $α$, $\t α$.
Then \eqref{eqn:invfun} takes the form
\begin{equation}
    \label{eqn:invfun-functor}
    φ_{\t V}∘χ=φ_V.
\end{equation}
We shall use this form later.
\begin{remark}
    The notion of an invariant function $T→ℝ$ should not be confused with the notion of a functional invariant.
    If $k=\codim T$, then $J_V$ consists of one point, and an invariant function provides a numeric invariant.
\end{remark}

\subsection{Factory of invariant functions}\label{sub:factory}

\subsubsection{A special polycycle of higher codimension}\label{subsub:poly-gamma-e-i}
Consider a normalized vector field $v_0$ with a hyperbolic polycycle $γ_e^0$ monodromic from the exterior.
Suppose that $γ_e^0$ includes a polycycle $γ_i^0$ monodromic from the interior, see \autoref{fig:gen}.

In order to introduce a condition analogous to \eqref{eqn:char}, we need the following definition.
\begin{definition}
    \label{def:carp}
    The \emph{characteristic number} of a hyperbolic polycycle $γ$ is the product of the characteristic numbers of all the saddles met during one turn along the polycycle;
    it is denoted by $λ(γ)$.
\end{definition}
Suppose that
\begin{align}
    \label{eqn:wind}
    λ(γ_e^0)&>1,&
    λ(γ_i^0)<1.
\end{align}
\begin{example}
    \label{ex:eqn-wind-char}
    The characteristic numbers of the polycycles $γ$ and $l$ from \autoref{sec:family} are equal to $λ²μ$ and $λ$, respectively.
    So, in this case \eqref{eqn:wind} becomes \eqref{eqn:char}.
\end{example}

Polycycles $γ_e^0$, $γ_i^0$ will be called \emph{special} for future references.

\subsubsection{Class \texorpdfstring{$\hat\bfT$}{T̂}}\label{subsub:class-hbfT}
Let us now construct a class $\hat\bfT=\hat\bfT(γ_e^0,γ_i^0)$.
Vector fields $v∈\hat\bfT$ are required to have polycycles $γ_i(v)⊂γ_e(v)$ such that the couple $(γ_e, γ_i)$ is homeomorphic to $(γ_e^0, γ_i^0)$.
At this spot, for the abuse of notation, we skip the expression $(v)$ that indicates the dependence of different objects on $v$.
Assume also that $v$ has no saddle connections other than the edges of $γ_e$.

Moreover, let $E$ and $I$ be hyperbolic saddles outside $γ_e$ and inside $γ_i$ respectively.
Suppose that there exists a smooth arc $Γ$ that connects $E$ and $I$ and is transversal to $v$ in its interior points;
let $O$ be the unique intersection point $Γ∩ γ_i$, and $l ⊂ γ_i$ be the saddle connection, edge of $γ_i$, that contains $O$.
Denote by $Γ^+$ and $Γ^-$ half open arcs of $Γ$ between $O$ and $E$ ($E$ excluded), and between $O$ and $I$ ($I$ excluded) respectively.
Note that the germs of the monodromy maps
\begin{align}
    \label{eqn:D-e-i}
    Δ_e&:(Γ^+,O)→(Γ^+,O), & Δ_i&:(Γ^-,O)→(Γ^-,O)
\end{align}
along the polycycles $γ_e$, $γ_i$ are well defined.
Suppose that the germs $Δ_e$, $Δ_i^{-1}$ may be extended as global Poincaré maps of the vector field $v$:
\begin{align*}
    Δ_e&: Γ^+→Γ^+, &Δ_i^{-1}&: Γ^- → Γ^-
\end{align*}
that have no fixed points except for $O$; both maps are into.
Then $O$ is an attracting fixed point of $Δ_e$ and of $Δ_i^{-1}$;
this follows from \eqref{eqn:wind}, see \autoref{rem:wind} below.

Let $\hat\mcH$, cf. \autoref{subsub:class-bfT}, be the class of vector fields $v$ such that
\begin{itemize}
    \item all the singular points and polycycles of the vector field $v$ are hyperbolic;
    \item products of all the characteristic values of the saddles of $v$ taken in the powers $0$, $1$, $2$, where at least one power is non-zero, are different from $1$.
\end{itemize}
Clearly, $\hat\mcH$ is open and dense in $\Vect(S^2)$.
We assume that $v∈\hat\mcH$.
We also assume that the following analog of \autoref{as:*} from \autoref{subsub:class-bfT} holds for the class $\hat\bfT$.
\begin{assumption}
    \label{as:func}
    For any polycycle $γ$ of a vector fields $v∈\hat\bfT$ there is no more than one saddle whose separatrix winds toward $γ$ in the positive or negative time.
\end{assumption}
This completes the definition of the class~$\hat\bfT$.
\begin{proposition}
    [cf. \autoref{prop:bfT-dist}]
    \label{prop:hbfT-dist}
    Consider two orbitally topologically equivalent vector fields $v, \t v∈\hat\mcH$.
    If $v∈\hat\bfT$, then $\t v∈\hat\bfT$.
    Moreover, a homeomorphism provided by \autoref{def:oteq} sends the polycycles $γ_e(v)$, $γ_i(v)$, and the saddles $E(v)$, $I(v)$ to similar objects for $\t v$.
\end{proposition}
\begin{proof}
    Let $H$ be a homeomorphism that links the phase portraits of $v$ and $\t v$.
    Similarly to the proof of \autoref{prop:bfT-dist}, the polycycles $γ_e(\t v)=H(γ_e(v))$, $γ_i(\t v)=H(γ_i(v))$ and the saddles $\t E=H(E)$, $\t I=H(I)$ satisfy all the requirements of the class $\hat\bfT$, possibly except for inequalities \eqref{eqn:wind}.
    Let us prove these inequalities.

    Suppose that $λ(γ_e(\t v))≤1$.
    Note that each saddle is met at most twice during one turn along $γ_e(\t v)$, hence $λ(γ_e(\t v))$ is a product of characteristic numbers of saddles of $\t v$ taken in the powers $0$, $1$, $2$.
    By definition of $\hat\mcH$, this implies $λ(γ_e(\t v))≠1$, hence $λ(γ_e(\t v))<1$.
    By \autoref{rem:wind}, this contradicts the property of the saddle $E$, namely, that a separatrix of $E$ winds towards $γ_e(\t v)$, and proves the left inequality in \eqref{eqn:wind}.
    One can prove the right inequality in \eqref{eqn:wind} in the same way.
\end{proof}

\begin{figure}
    \centering\includegraphics[scale=0.8]{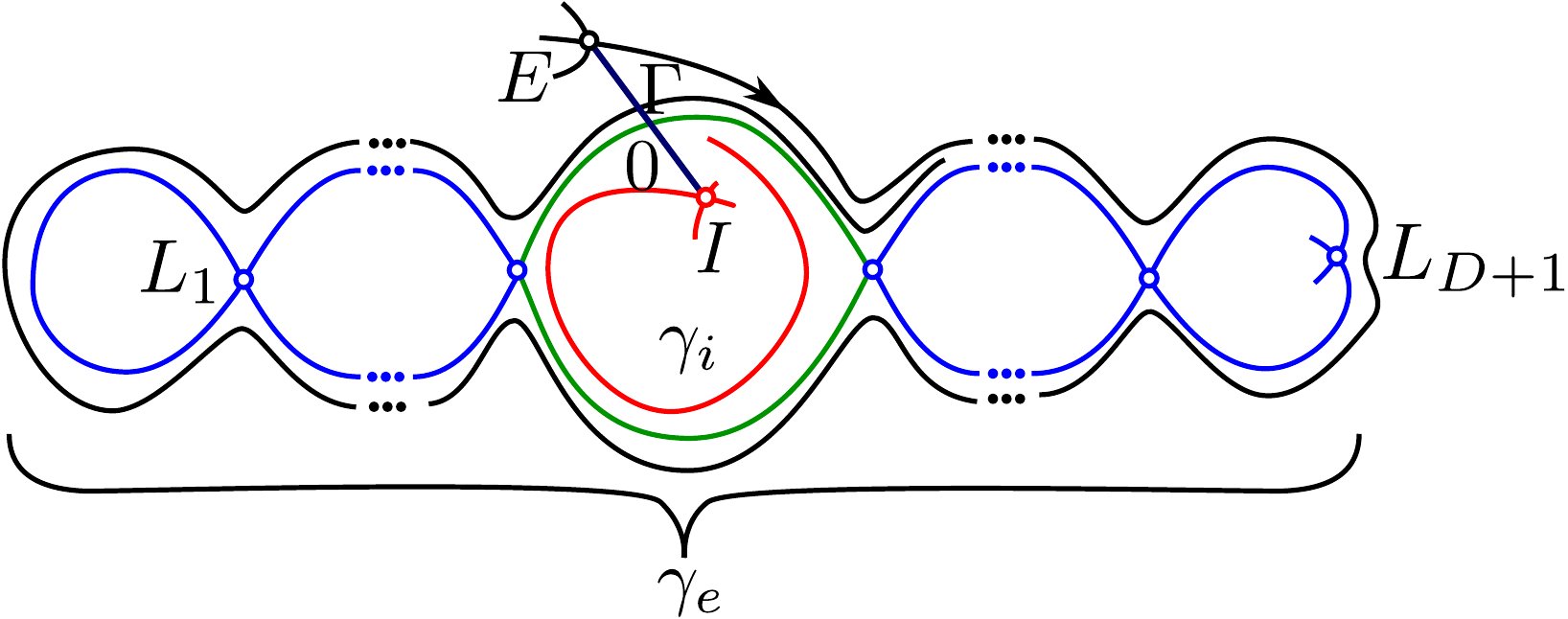}
    \caption{Factory of invariant functions}
    \label{fig:gen}
\end{figure}

Let $φ$ be the function given by
\begin{equation}
    \label{eqn:infuu}
    φ(v)= - \frac{\log λ(γ_i(v))}{\log λ(γ_e(v))},\quad φ\colon\hat\bfT→ℝ_+.
\end{equation}
Fix $k ≥ \codim \hat\bfT$.
Consider the class $\hat \bfT^\trans_k$.
Recall that according to \autoref{subsub:unfoldings} this is the class of $k$-parameter unfoldings of vector fields of class $\hat\bfT$.

\begin{theorem}
    \label{thm:invfun}
    The function $φ$ is invariant for the families of class $\hat \bfT^\trans_k$ in sense of \autoref{def:inv}.
\end{theorem}
\begin{proof}
We will prove \autoref{thm:invfun} in the current \autoref{sub:factory}.
The proof follows the same lines as for \autoref{thm:unst1}.

\subsubsection{Some bifurcations in the families of class \texorpdfstring{$\hat \bfT^\trans_k$}{T̂\textasciicircum ⫛\_k}}
Consider a local family $V = \set{ v_α |α ∈ (ℝ^k,0)} $ of class $\hat \bfT^\trans_k$.
By definition, $v_0 ∈ \hat\bfT$.
Let $γ_e$, $γ_i$ be two polycycles and let $l$ be a saddle connection of $v_0$ described in the previous subsection.
As in \autoref{sub:inv-functions}, put $J_V=\hat \bfT∩V$, $d ≔ \dim J_V$.

Let $l_j$, $j=1,…,N$, be the edges of $γ_e$, $l=l_1$.
Let $σ_j$ be the corresponding separatrix splitting parameters defined as in \autoref{sub:gencond}.
Assume that the cross section $Γ_1$ used to define $σ_1$ coincides with $Γ$ near $Γ∩l$.
Put $ε=σ_1$.

Let $ℰ = ℰ_{V,l}$ be the $(d + 1)$-dimensional subfamily of $V$ defined by
\begin{equation}
    \label{eqn:def-E-dim-N}
    ℰ_{V,l}=V∩\set{σ_2=\dots=σ_N=0}.
\end{equation}
Obviously, $J_V ⊂ ℰ$.

Let $η∈ (ℝ^d,0)$ be a coordinate on $B(J_V)$ smoothly extended to $B(ℰ)$, $η(v_0) = 0$.
By the transversality of $V$ to $\hat \bfT$, $B(ℰ)$ is a germ of a smooth submanifold of $B(V)$, and $α=0$ is a~non-critical point of the function $ε$ on $B(ℰ)$.
Then $β≔(η, ε)$ may be taken as a coordinate (new parameter) on $B(ℰ)$ near $0$.
By definition of $ε$, $J_V=\set{v_α|v_α∈ℰ, ε(α)=0}$.

Vector fields $v_β∈ ℰ$ have sparkling saddle connections between the saddles $E(β)$, $M(β)$ on one hand, and $L(β)$, $I(β)$ on the other hand (exterior and interior connections respectively).
Let $Γ (β)$ be a cross section that plays for $v_β$ the same role as $Γ$ plays for $v = v_0$.
Namely, $Γ (β)$ connects $E (β)$ and $I(β)$, and is transversal to $v_β$ at its interior points.
Moreover, $Γ (β)$ continuously depends on $β$, and $Γ (0) = Γ$.
We may assume that $Γ (β)$ coincides with $Γ $ outside some small neighborhoods of $E$ and $I$.
\begin{definition}
    [cf. \autoref{def:n-turns-numer}]
    \label{def:connect1}
    A connection $E(β)M(β)$ (respectively, $L(β)I(β)$) \emph{makes} $n$ (respectively, $m$) \emph{turns around $γ_e$} (around $γ_i$) provided that it intersects the arc $Γ(β)$ with endpoints excluded at exactly $n$ points (resp., $m$ points).
\end{definition}

The following analog of \autoref{lem:asym} holds:
\begin{lemma}
    [Generalized Asymptotic Lemma]
    \label{lem:asym1}
    Consider the family $ℰ = \set{v_β}$ described above, $β = (η, ε)$.
    For $m$, $n$ sufficiently large, the following holds.
    In some neighborhood of $0$ in $ℝ^d$, non depending on $m$ and $n$, there exist uniquely determined positive functions $ε = i_m(η)$, $ε = e_n(η)$ such that the vector field $v_β$, $β = (η, i_m(η))$ ($β=(η, e_n(η))$) has an interior (exterior) saddle connection that makes $m$ (respectively $n$) turns around $γ_i$ (around $γ_e$).
    The functional sequences $(i_m)$, $(e_n)$ decrease monotonically and satisfy the equations:
    \begin{align}
        \label{eqn:im1}
        \log (-\log i_m(η)) &= - m\log λ_i(η) + O(1),\\
        \label{eqn:en1}
        \log (-\log e_n(η)) &= n\log λ_e(η) + O(1),
    \end{align}
    where $λ_i(η) = λ(γ_i(v_β))$, $λ_e(η) = λ(γ_e(v_β))$, $β = (η, 0)$,
    and the remainder terms are uniformly bounded in $η∈(ℝ^d,0)$.
\end{lemma}
\begin{remark}
    \label{rem:L3-hence-L1}
    When $k=\codim\hat\bfT$, the base of $J_V$ is one point $\set{0}$, and functions $i_m$, $e_n$ become numbers.
    In this case \autoref{lem:asym1} implies \autoref{lem:asym}, setting $\hat\bfT=\bfT$, $k=3$.
\end{remark}

This lemma is proved in \autoref{sec:aux} together with \autoref{lem:asym}.

\subsubsection{Two equivalent families of class \texorpdfstring{$\hat \bfT^\trans_k$}{T̂\textasciicircum ⫛\_k}}
In what follows, we suppose that $k>\codim\hat \bfT$.
If $k=\codim\hat\bfT$, then \autoref{thm:invfun} is proved in the same way as \autoref{thm:unst1}.
Consider two moderately topologically equivalent local families of class $\hat \bfT^\trans_k$:
$V =\{v_α\}$ and $\wt V=\{\t v_{\t α}\}$.
Let $J_V = V∩\hat \bfT$, $J_{\t V}=\t V∩\hat\bfT$.
Let $H=(h, H_α)$, $h(0)=0$, be a moderate equivalence~\eqref{eqn:conj} that links these two families.
By \autoref{prop:hbfT-dist}, the class $\hat\bfT$ is topologically distinguished in its neighborhood, hence $h(B(J_V))=B(J_{\t V})$.

Let $l=l_1$, $l_j$ for $j=2,…,N$, and $ℰ$ be the same as in the previous subsection.
Put $\t l = H_0 (l)$, $\t l_j = H_0(l_j)$.
Let $\t σ_j:B(\t V)→ℝ$ be the separatrix splitting parameters for $\t l_j$, $\t ε=\t σ_1$.
Let $\t ℰ=\t V∩\set{\t σ_2=⋯=\t σ_N=0}$.
Similarly to \autoref{subsub:num}, \autoref{prop:dist} implies that $h$ sends $\set{σ_j=0}$ to $\set{\t σ_j=0}$, hence $h$ sends $ℰ$ to $\t ℰ$.

Let $i_m$ and $e_n$ be the functions corresponding to sparkling saddle connections in the family $ℰ$ introduced in \autoref{lem:asym1}.
Denote by $\mcI_m ⊂ ℰ$, $ℰ_n ⊂ ℰ$ the graphs of the functions $i_m, e_n\colon B(J_V)→ℝ_+$, respectively, defined for $r$ small, and $m$, $n$ large enough, and by $\mcI_m^+$ and $ℰ_n^+$ the domains $ε > i_m(η)$, $ε > e_n(η)$.
Let $\wt \mcI_m$, $\wt ℰ_n$, $\wt \mcI_m^+$, $\wt ℰ_n^+$ be the similar objects for $\wt V$.

The functional sequences $(i_m)$ and $(e_n)$ are monotonically decreasing by \autoref{lem:asym1}.
The sequence $(i_m)$ corresponds to saddle connections between the saddles $L(α)$, $I(α)$, where $α ∈B(ℰ)$, $ε(α) = i_m (η (α))$.
We ignore here the fact that this connection makes $m$ turns.
The vector field $v_α ∈ ℰ$ that has a saddle connection $L(α) I(α)$ corresponds to a vector field $\t v_α ∈\wt ℰ$ that has a saddle connection $\t L(\t α) \t I(\t α)$.
Hence, the map $h$ brings the germ of the union of graphs $⋃_{m≥m_0} \mcI_m$ at $0$ to the germ of the union of graphs $⋃_{\t m≥\t m_0}^∞ \wt \mcI_{\t m}$ at $0$.
A similar statement holds fo the unions $⋃_{n≥n_0}ℰ_n$ and $⋃_{\t n≥\t n_0}\t ℰ_{\t n}$.

The sequences of domains $(ℐ_m^+)$, $(ℰ_n^+)$, $(\wt ℐ_m^+)$, $(\wt ℰ_n^+)$ are monotonically increasing with respect to inclusion.
The homeomorphism $h$ respects this property.
Hence, there exist integer $a$ and $b$ such that
\begin{equation}
    \label{eqn:graphs}
    h ( \mcI_m) = (\wt \mcI_{m-a}),\quad
    h ( ℰ_m) = (\wt ℰ_{m-b}).
\end{equation}
\subsubsection{Relative density of two functional sequences}
Simlarly to \autoref{subsub:relative-density}, we introduce the counting functions
\[
    N_i(η, ε) = \# \set{m>m_0| ε ≤ i_m(η) }, \
    N_e(η, ε) = \# \set{n>n_0| ε ≤ e_n(η) }.
\]
In other words, $N_i(η, ε)=m$ for $β=(η, ε)∈ℐ_{m+m_0+1}∖ℐ_{m+m_0}$, and similarly for $N_e$.
In the same way, functions $\t N_i$ and $\t N_e$ are defined.

Define the \emph{relative density} of $(i_m)$ and $(e_n)$ as a limit
\[
    ν (\eta) =
        \lim_{\substack{(η', ε)→(η, 0)\\ε>0}} \frac{N_e(η', ε)}{N_i(η', ε)},
\]
if exists.
Then \autoref{lem:asym1} implies that
\begin{equation}
    \label{eqn:lim-Ni-Ne-func}
    \begin{aligned}
       ν(η) = \lim_{\substack{(η', ε)→(η, 0)\\ε>0}} \frac{N_e(η', ε)}{N_i(η', ε)}&=
        \lim_{\substack{(η', ε)→(η, 0)\\ε>0}} \frac{\log(-\log ε)+O(1)}{\log λ_e(η')}×\frac{-\log λ_i(η')}{\log(-\log ε)+O(1)}\\
        &=-\frac{\log λ_i(η)}{\log λ_e(η)}=φ(v_{η, 0}),
    \end{aligned}
\end{equation}
and a similar equality holds for $\t N_e$ and $\t N_i$.

\subsubsection{Invariance of the function \texorpdfstring{$φ$}{φ}}
We will deduce invariance of the function $φ$, see~\eqref{eqn:infuu}, from the properties~\eqref{eqn:graphs} and \eqref{eqn:lim-Ni-Ne-func}.
Consider a point $(η, 0)∈(ℝ^d, 0)$ and the corresponding point $(\t η, 0)=h(η, 0)∈(ℝ^d, 0)$.
Let us compute the limit $\displaystyle\lim_{\substack{(η, ε)→(η_0, 0)\\ε>0}} \frac{\t N_e(h(η, ε))}{\t N_i(h(η, ε))}$ in two ways.
On one hand, due to \eqref{eqn:graphs} and \eqref{eqn:lim-Ni-Ne-func},
\[
    \lim_{\substack{(η', ε)→(η, 0)\\ε>0}} \frac{\t N_e(h(η', ε))}{\t N_i(h(η', ε))}=
    \lim_{\substack{(η', ε)→(η, 0)\\ε>0}} \frac{N_e(η', ε)+\const}{N_i(η', ε)+\const}=
    \lim_{\substack{(η', ε)→(η, 0)\\ε>0}} \frac{N_e(η', ε)}{N_i(η', ε)}=
    φ(v_{η, 0}).
\]
On the other hand, $h$ is a homeomorphism that sends $ℰ^+$ to $\t ℰ^+$, hence
\[
    \lim_{\substack{(η', ε)→(η, 0)\\ε>0}} \frac{\t N_e(h(η', ε))}{\t N_i(h(η', ε))}=
    \lim_{\substack{(\t η', \t ε)→(\t η, 0)\\\t ε>0}} \frac{\t N_e(\t η', \t ε)}{\t N_i(\t η', \t ε)}=
    φ(\t v_{\t η, 0}).
\]
This implies the second part of \eqref{eqn:invfun}, so we proved \autoref{thm:invfun}, modulo \autoref{lem:asym1}.
\end{proof}

\subsection{Simple diagrams}\label{sub:simdiag}
Suppose that in the construction above, the polycycle $γ_e$ contains several polycycles $γ_i^j$, $j=1,…,D$, monodromic from the interior.
Suppose that assumptions of \autoref{thm:invfun} hold for each pair of polycycles $(γ_e, γ_i^j)$.
Then an unfolding of the corresponding vector field may have several invariant functions $φ_j$ simultaneously, one for each pair $(γ_e, γ_i^j)$.

Together with \eqref{eqn:invfun-functor}, this motivates the following definition.
\begin{definition}
    \label{def:simple}
    A \emph{simple diagram of rank $(d, D)$}, $d<D$, is an equivalence class of germs of smooth maps,
    \begin{equation}
        \label{eqn:diag}
        φ:(M_φ,0)→(ℝ^D,φ(0)),
    \end{equation}
    where $(M_φ, 0)$ is a germ of a $d$-dimensional manifold.
    Two germs \eqref{eqn:diag} are equivalent, $φ\sim \t φ$, provided that there exists a homeomorphism $χ: (M_φ,0) →(M_{\t φ},0)$ such that
    \begin{equation}
        \label{eqn:equiv2}
        \t φ \circ χ = φ.
    \end{equation}
    Denote by $[φ]$ the equivalence class that contains $φ$.
\end{definition}
\begin{remark}
    We write $(M_φ, 0)$ instead of $(ℝ^d, 0)$ to underline that the there are no canonical coordinates in the domain of $φ$, except for the origin.
    On the other hand, the coordinates in the image of the map~\eqref{eqn:diag} are well defined.
    These are the values of the functions $φ_j$.
\end{remark}

Generic germ of map~\eqref{eqn:diag} has rank $d$.
Two maps of rank $d$ are equivalent if and only if the germs of their images coincide.
Fix a decomposition $ℝ^D=ℝ^d\oplus ℝ^{D-d}$.
The first space is spanned by the first $d$ axis, the second one by the last $D - d$ ones.

\begin{remark}
    \label{rem:real}
    For a generic smooth map $φ$, see \eqref{eqn:diag}, its image is a graph of a map $ f: (ℝ^d,a)→(ℝ^{D-d}, b)$, $(a,b) = φ(0)$.
    This map is called the \emph{modulus} of $φ$.
    Two generic smooth maps of the form \eqref{eqn:diag} are equivalent iff their moduli coincide.
    So, we may call $f$ the modulus of the simple diagram $[φ]$.
    Any such modulus may be realized by some simple diagram.
\end{remark}
For instance, a simple diagram $[φ]$ with generic $φ$~\eqref{eqn:diag} for $d=1$, $D=2$ is characterized by a~germ of a smooth function $(ℝ, a)→(ℝ, b)$, cf. \autoref{thm:func-unst}.

\autoref{def:inv} in the form \eqref{eqn:invfun-functor}, and \autoref{def:simple} imply the following proposition.
\begin{proposition}
    \label{prop:many-diag}
    Let $T$ be a Banach submanifold of $\Vect(S^2)$.
    Suppose that the moderate topological classification of families of class $T^\trans_k$ possesses $D>d=k-\codim T$ invariant functions $φ_j:T→ℝ$, $j=1,…,D$.
    Then the simple diagram $[φ_V]$, where $φ_V$ is given by $φ_V=(φ_1, …, φ_D)∘J_V$, or equivalently,
    \begin{align}
        \label{eqn:phi-V}
        φ_V(α)&=(φ_1(v_α), …, φ_D(v_α)),&
        φ_V&:(B(J_V), 0)→(ℝ^D, φ_V(0)),
    \end{align}
    is an invariant of this classification.
\end{proposition}

The following theorem is a local version of \autoref{thm:fuinv}.
\begin{theorem}
    \label{thm:fuinv-loc}
    For any two positive integers $d<D$ there exists a Banach submanifold $\bfT_D⊂\Vect(S^2)$ of codimension $2D + 1$ such that the moderate topological classification of families of class $\left(\bfT_D\right)^\trans_k$ has a simple diagram $[φ]$ of rank $(d, D)$ as an invariant.
    Any diagram with positive $φ(0)$ may be realized as an invariant of some family $V∈\left(\bfT_D\right)^\trans_k$.
\end{theorem}
We prove this theorem in the next two subsections, then we deduce \autoref{thm:fuinv} from it.
The proof is based on \autoref{thm:invfun} and \autoref{prop:many-diag}.

\subsection{Six parameter local families with a functional invariant}

In this section we prove \autoref{thm:fuinv-loc} for $d = 1$, $D = 2$.
Let us describe the class $\bfT_2$.

Consider a vector field $v_0$ with a hyperbolic polycycle $γ_e(v_0)$ with three vertexes:
saddles $L_1$, $L_2$, $L_3$, and five edges:
a separatrix loop of the saddle $L_1$, and four time oriented connections: $L_jL_{j+1}$, $L_{j+1}L_j$, $j = 1,2$, see \autoref{fig:six}.
This polycycle includes three smaller polycycles:
the separatrix loop $γ_i^1$ of the saddle $L_1$,
the polycycle $γ_i^2$ formed by the connections $L_1L_2$, $L_2L_1$,
and the polycycle $γ_h$ formed by the connections $L_2L_3$, $L_3L_2$.
Each of these three smaller polycycles is located outside of the other two.
The polycycle $γ_h$ is a polycycle “heart”, see \autoref{sub:spec1}.

\begin{figure}
    \centering\includegraphics[scale=0.8]{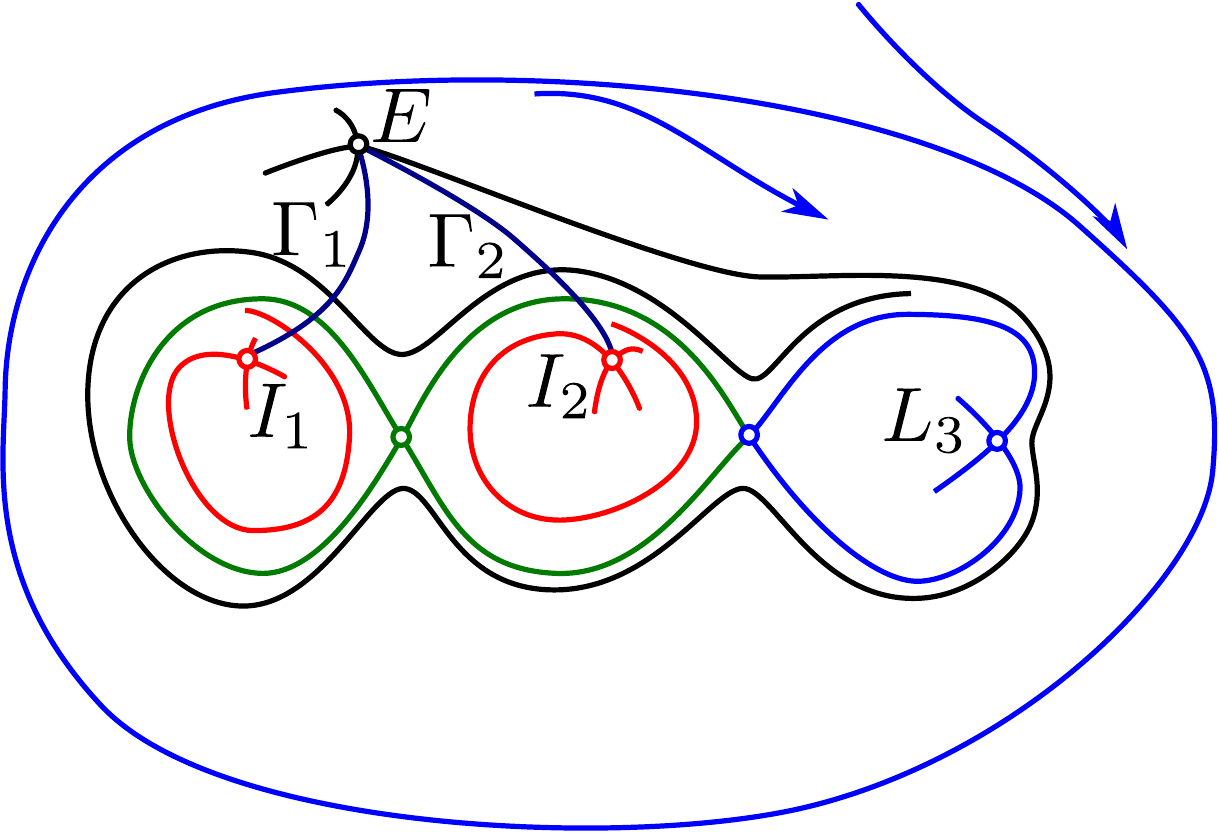}
    \caption{
        Phase portrait of a vector field of class~$\bfT_{2,1}$.
        The phase portrait of a vector field of class $\bfT_2$ may lack the outer semi-stable limit cycle.
    }
    \label{fig:six}
\end{figure}

The polycycle $γ_e$ is therefore monodromic from the exterior,
and the polycycles $γ_i^1$, $γ_i^2$ are monodromic from the interior.
Let $λ_j$ be the characteristic number of $L_j$.
Then the characteristic numbers of the polycycles $γ_i^1$, $γ_i^2$, $γ_e$ are given by
\begin{align*}
    λ(γ_i^1) &= λ_1,&
    λ(γ_i^2) &= λ_1λ_2,&
    λ(γ_e) &= λ_1^2λ_2^2λ_3.
\end{align*}
Suppose that both pairs $(γ_e, γ_i^1)$ and $(γ_e, γ_i^2)$ satisfy \eqref{eqn:wind}, that is:
\begin{align}
    \label{eqn:lg}
    λ_1^2λ_2^2λ_3 &> 1,&
    λ_1&<1,&
    λ_1λ_2&<1.
\end{align}

This completes the description of the polycycle $γ_e$.
Note that both pairs $(γ_e, γ_i^1)$ and $(γ_e, γ_i^2)$ satisfy the properties formulated in \autoref{subsub:poly-gamma-e-i}.
Hence, both classes $\hat\bfT(γ_e, γ_i^1)$ and $\hat\bfT(γ_e, γ_i^2)$ are well defined.
Let $\bfT_2$ be the intersection of these two classes.
\[
    \bfT_2=\hat\bfT(γ_e,γ_i^1)∩\hat\bfT(γ_e,γ_i^2).
\]
This class has codimension $5$.

\begin{remark}
    For $v∈\bfT_2$, the same saddle $E$ plays the role of $E$ from \autoref{subsub:class-hbfT} for both pairs $(γ_e, γ_i^1)$ and $(γ_e, γ_i^2)$, though the saddles $I_1$, $I_2$ inside the polycycles $γ_i^1$, $γ_i^2$ are of course different, see \autoref{fig:six}.
    In a neighborhood of $v∈\bfT_2$, we have $\bfT_2=\hat\bfT(γ_e, γ_i^1)=\hat\bfT(γ_e, γ_i^2)$.
    Indeed, a vector field in $\hat\bfT(γ_e, γ_i^1)$ has the polycycle $γ_e$ preserved, hence both $γ_i^1$ and $γ_i^2$ are preserved, and a small perturbation cannot destroy \eqref{eqn:lg}.
\end{remark}

Let $V$ be a local family of the class $\left(\bfT_2\right)^\trans_6$, i.e., $v_0∈\bfT_2$, and $V$ is transverse to $\bfT_2$ at $v_0$.
As $\bfT_2=\hat\bfT(γ_e,γ_i^1)$ near $v_0$, we can apply \autoref{thm:invfun} to $γ_i=γ_i^1$, and get an invariant function $φ_1$:
\begin{align*}
    φ_1&:\bfT_2→ℝ,&
    φ_1(w)&=-\frac {\log λ(γ_i^1(w))}{\log λ(γ_e(w))}.
\end{align*}
We can apply the same theorem to $γ_i=γ_i^2$, and get an invariant function
\begin{align*}
    φ_2&:\bfT_2→ℝ,&
    φ_2(w)&=-\frac {\log λ(γ_i^2(w))}{\log γ_e(w)}.
\end{align*}

Therefore, due to \autoref{prop:many-diag}, the simple diagram $[φ_V]$ given by $φ_V=(φ_1, φ_2)∘J_V$ is an invaraint of the moderate topological classification of the families of class $\left( \bfT_2 \right)^\trans_6$.

The only restriction on the functions $λ_j(w_η)$ is \eqref{eqn:lg}, hence any simple diagram $[φ]$ of rank $(1, 2)$ with positive $φ(0)$ can be realized as $[φ_V]$.
This proves \autoref{thm:fuinv-loc} for $d = 1$, $D = 2$, modulo \autoref{lem:asym1}.

\subsection{Simple diagrams as functional invariants}

The general version of \autoref{thm:fuinv-loc} is proved in the same way.
Fix $0 < d < D$.
Consider a vector field $v_0$ having a polycycle $γ_e$ with $D + 1$ hyperbolic saddles $L_1, \dots , L_{D+1}$, and $2D + 1$ time oriented edges: the loop of $L_1$ and the connections $L_jL_{j+1}$; $L_{j+1}L_j$.
The edges $L_DL_{D+1}$; $L_{D+1}L_D$ form a polycycle of the type “heart”, see \autoref{fig:gen}.
The characteristic numbers of the saddles $L_j$ are $λ_j$.
The polycycle $γ_i^1$ is the loop of $L_1$, and the polycycles $γ_i^j$, $j = 2, \dots , D-1$ are formed by the edges $L_{j-1}L_j$; $L_jL_{j-1}$.
Their characteristic numbers are:
\begin{align*}
    λ (γ_i^1) &= λ_1,&
    λ(γ_i^j) &= λ_{j-1}λ_j,\,j > 1.
\end{align*}
The characteristic number of the “large” polycycle is
\[
    λ (γ_e) = λ_{D+1}\prod_1^Dλ_j^2.
\]
We require:
\begin{align}
    \label{eqn:lg1}
    λ(γ_i^j) &< 1, & λ(γ_e) &> 1.
\end{align}

Let $\bfT_D$ be given by
\[
    \bfT_D=\hat\bfT(γ_e, γ_i^1)∩…∩\hat\bfT(γ_e, γ_i^D).
\]
In a neighborhood of any vector field $v_0∈\bfT_D$, we have $\bfT_D=\hat\bfT(γ_e, γ_i^j)$, and $\bfT_D$ is a Banach submanifold of the space $\Vect(S^2)$ of codimension $2D+1$.

Let $V$ be a local family of class $\left(\bfT_D\right)^\trans_k$, $k=d+\codim \bfT_D$.
As $\bfT_D=\hat\bfT_{γ_e,γ_i^j}$ near $v_0$, \autoref{thm:invfun} implies that functions
\begin{align*}
    φ_j&: \bfT_D→ℝ,&
    w &↦-\frac {\log λ (γ_i^j(w))}{\log λ (γ_e(w))}
\end{align*}
are invariant in sense of \autoref{def:inv}.

\autoref{prop:many-diag} implies that the simple diagram $[φ_V]$ given by $φ_V=(φ_1, …, φ_D)∘J_V$ is a functional invaraint of moderate topological equivalence of families $V∈\left(\bfT_D\right)^\trans_k$.

As the characteristic numbers $λ_j(w)$ may be chosen as arbitrary functions of parameters with the only restrictions~\eqref{eqn:lg1}, the simple diagram~$[φ_V]$ may be arbitrary with positive $φ(0)$.
This proves \autoref{thm:fuinv-loc}, modulo~\autoref{lem:asym1}.

\subsection{Non-local families with functional invariants}
In this section we deduce \autoref{thm:fuinv} from \autoref{thm:fuinv-loc}.

Let $V$ and $\t V$ be two local families of class $\left(\bfT_D\right)^\trans_k$, and $\bfV$, $\t\bfV$ be their representatives.
Note that moderate topological equivalence of non-local families $\bfV$, $\t\bfV$ does not imply the moderate topological equivalence of $V$ and $\t V$.
Indeed, the corresponding homeomorphism of the bases may send $0$ to another point of $\t\bfV∩\bfT_D$, but \autoref{def:local} requires it to send zero to zero.

In order to avoid this problem, we introduce a class $\bfT_{D,d}⊂\bfT_D$ of codimension $k$.
Additionally to the parts of the phase portrait required by definition of $\bfT_D$, a vector field $v ∈ \bfT_{D,d}$ has $d$ semistable limit cycles of multiplicity~$2$ separating the polycycle $γ_e$ and the limit cycle at infinity.
Clearly, $\codim\bfT_{D,d}=d+\codim\bfT_D=k$.

Let $\bfT_{D,d}^\trans$ be the class of non-local $k$-parametric families of vector fields $V⊂\hat\mcH$ that intersect $\bfT_{D,d}$ at a single point, and contain no other vector fields orbitally topologically equivalent to vector fields of class $\bfT_{D,d}$.
Consider two moderately topologically equivalent families $\bfV, \t\bfV∈\bfT_{D,d}^\trans$.
Let us shift coordinates in these families so that $v_0=\bfV∩\bfT_{D,d}$, $\t v_0=\t\bfV∩\bfT_{D,d}$.
Let $h$ be the homeomorphism of the bases from \eqref{eqn:conj}.
Due to the definition of $\bfT_{D,d}^\trans$, $h$ sends zero to zero, hence the \emph{local} families $(\bfV, v_0)$ and $(\t\bfV, \t v_0)$ are moderately topologically equivalent.
Note that these families belong to $\left(\bfT_{D,d}\right)^\trans_k$, hence they belong to $\left(\bfT_D\right)^\trans_k$ as well.

Finally, let $[φ_V]$ be the simple diagram provided by \autoref{thm:fuinv-loc}.
For a non-local family $\bfV$, let $V$ be its germ at $\bfV∩\bfT_{D,d}$.
Then the simple diagram $[\hat{φ}_\bfV]=[φ_V]$ is an invariant of moderate topological equivalence of families $\bfV∈\bfT_{D,d}^\trans$.
Therefore, families of class $\bfT_{D,d}^\trans$ have simple diagrams \eqref{eqn:diag} as invariant of moderate topological equivalence, and any diagram with $φ(0)∈ℝ_+^D$ may be realized as the invariant of a family $\bfV∈\bfT_{D,d}^\trans$.
Due to \autoref{rem:real}, an open dense subset of $\bfT_{D,d}^\trans$ has a smooth map $(ℝ_+^d, a)→(ℝ_+^{d'}, b)$ as a functional invariant of the moderate topological equivalence.

This completes the proof of \autoref{thm:fuinv} modulo \autoref{lem:asym1}.
\section{Asymptotics of sparkling saddle connections}\label{sec:aux}

In this section we prove Lemmas~\ref{lem:asym} and~\ref{lem:asym1}, and thus complete the proofs of Theorems~\ref{thm:func-unst}−\ref{thm:unst1} and~\ref{thm:fuinv-loc}.

\subsection{General asymptotics lemma and connection equation}
\label{sub:gen}

\autoref{lem:asym} follows from \autoref{lem:asym1}, see \autoref{rem:L3-hence-L1}, which, in turn contains two statements:
one about the sparkling saddle connections related to the polycycle $γ_i$;
another one about those related to the polycycle $γ_e$, see~\autoref{sub:spec1}.
These are two particular cases of one general lemma that will be stated now.

Consider a vector field $v$ having a hyperbolic monodromic polycycle $γ$.
Let $λ(γ) $ be its \emph{characteristic number}, see \autoref{def:carp}.

Mark one edge of $γ$, say from $L$ to $M$, and a point $O$ on it.
Suppose that $λ (γ) < 1$, and there exists a hyperbolic saddle $I$ whose incoming separatrix $W_I^s$ winds toward $γ$ in the negative time.
Let $Γ$ be a cross section through $O$ with one endpoint $I$ transversal to $v$ everywhere except for $I$.
A germ of a monodromy map $Δ_γ: (Γ ,O)→(Γ ,O)$ is well defined by assumption that $γ$ is monodromic.
Suppose that the germ $Δ_γ$ may be extended to the monodromy map $Δ_γ: I'O→IO$,
where $I'$ is the first intersection point of $Γ $ and the separatrix $W^s_I$ ran from $I$ in the negative time,
see \autoref{fig:int-cross}.

\begin{figure}[hbt]
    \centering
    \subcaptionbox{\label{fig:int-cross}}{\includegraphics[scale=0.5]{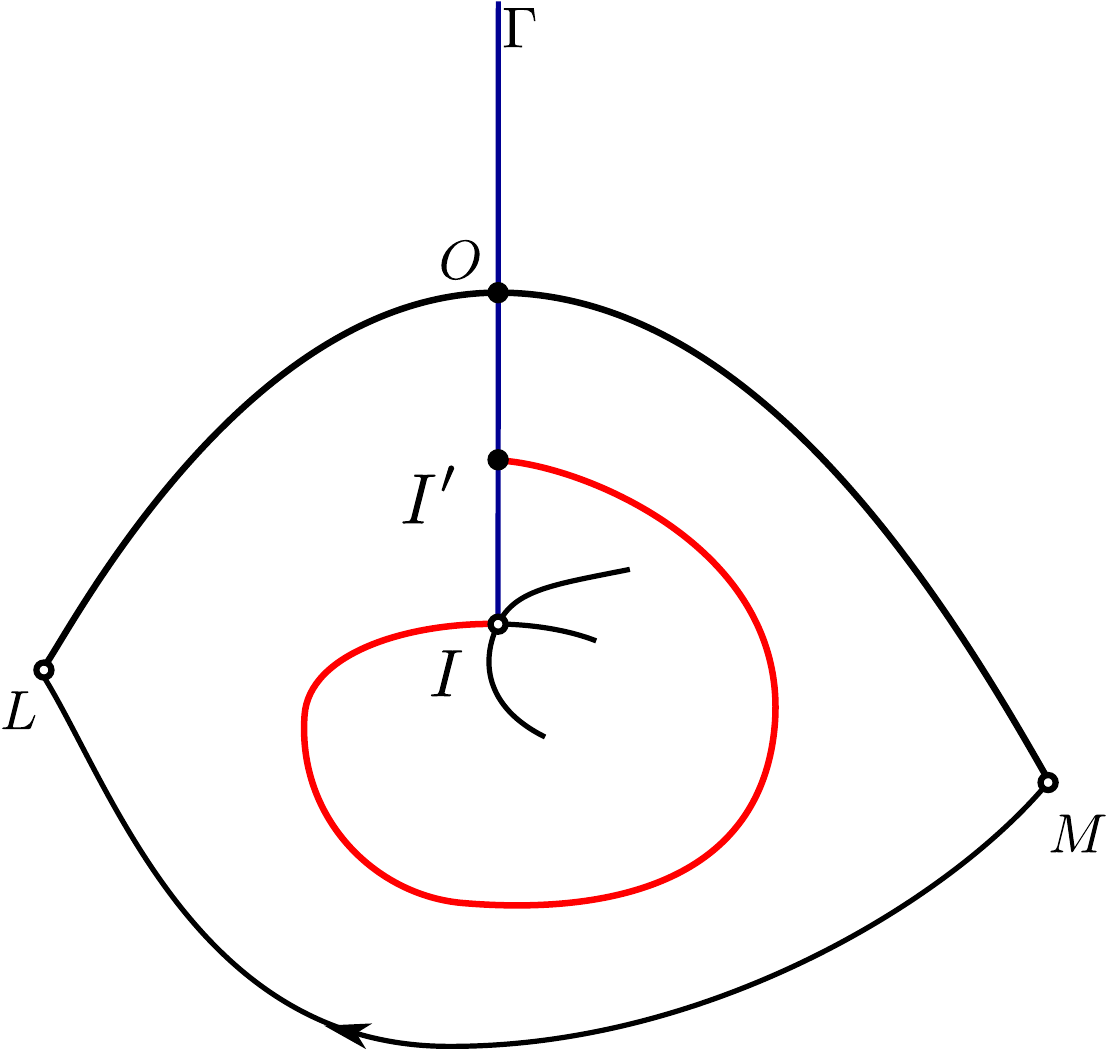}}
    \subcaptionbox{\label{fig:conn-eq}}{\includegraphics[scale=0.5]{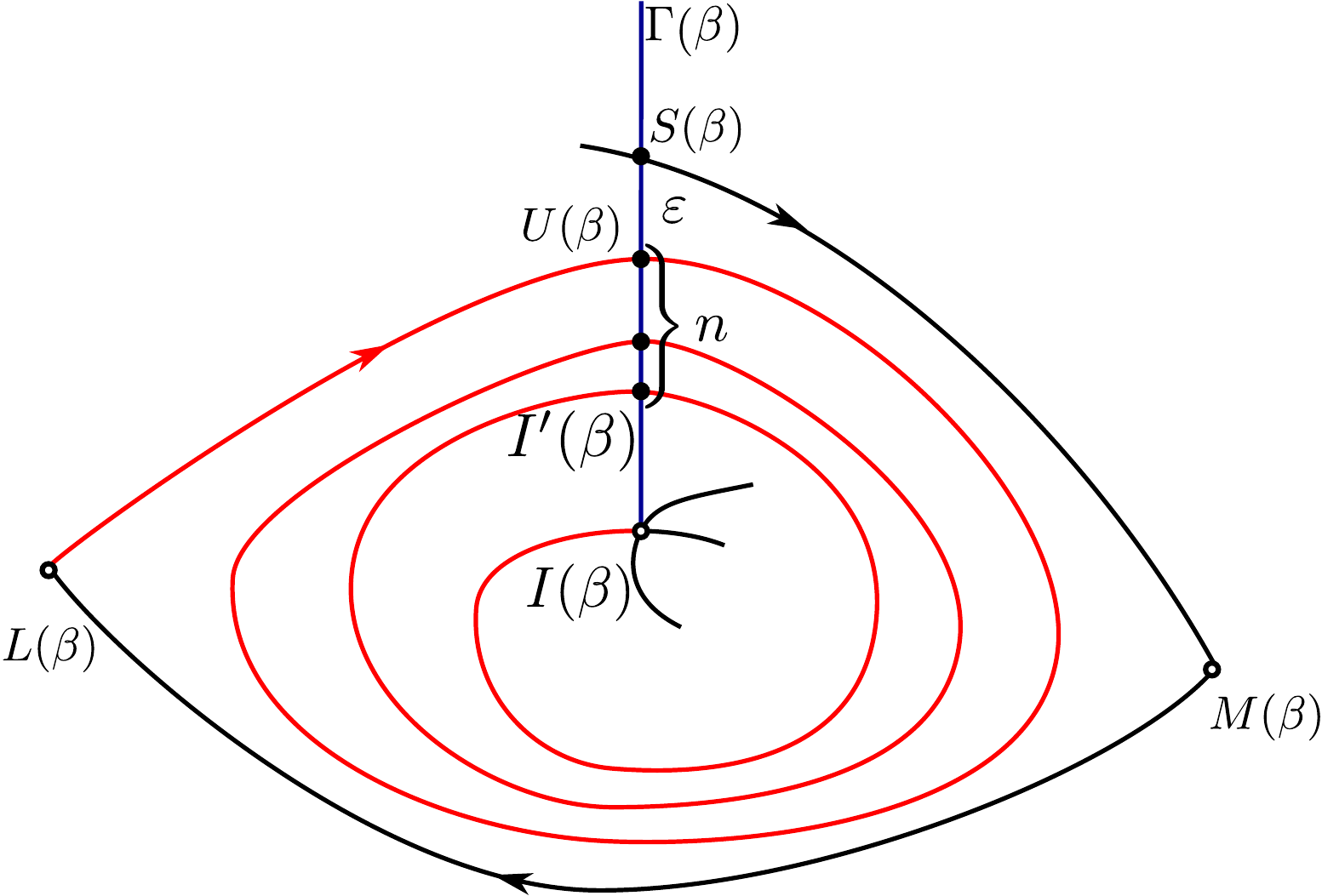}}
    \caption{Geometry of the connection equation}
    \label{fig:conn-eq-both}
\end{figure}

Let $ℰ $ be a $(d + 1)$-parameter (non-generic) family of vector fields,
\begin{equation}
    \label{eqn:famee}
    ℰ = \set{v_β|β∈(ℝ^{d+1},0)} , \ β= (η , ε ), \ η∈(ℝ^d,0), \ ε∈(ℝ ,0), \ v_0 = v.
\end{equation}
Suppose that for $ε = 0$, $β= (η, 0)$ the vector field $v_β$ has a polycycle $γ(η )$ continuously depending on $η$; $γ(0) = γ$.
Let $L(β)$, $M(β)$, $I(β)$ be hyperbolic saddles continuously depending on $β$ and coinciding with $L$, $M$, $I$ for $β= 0$.
For $ε \ne 0$ all the edges of the polycycle $γ$ stay unbroken except for the connection between $L(β)$ and $M(β)$.

For small $η$, the polycycle $γ(η )$ is hyperbolic;
let $λ (η ) = λ (γ(η)) < 1$ be its characteristic number.
Let $U(β)$ ($S(β)$) be the first intersection point of the unstable separatrix of $L(β)$ (stable separatrix of $M(β)$) with $Γ$, continuous in $β$;
$U(0) = S(0) = O$.
Consider a family of cross sections $Γ (β)$ that connect $S(β)$ and $I(β)$;
assume that $Γ(β)$ coincides with $Γ$ outside a small neighborhood of $I$.
Suppose that for small $β$, the map $Δ_γ$ may be extended in $β$ as a map of an arc in $Γ$ to an arc in $Γ_β$:
\[
    Δ_{γ,β}: I'(β)S(β)→I(β)U(β),
\]
see \autoref{fig:conn-eq}.

As before, say that $v_β$ has an $n$-winding connection $L(β)I(β)$ provided that the outcoming separatrix of $L(β)$ crosses the arc $Γ(β)$ with the point $I(β)$ excluded at exactly $n$ points, and enters $I(β)$.

Let $x$ be any smooth chart on $Γ $ positive on the semitransversal where the monodromy map $Δ_{γ, 0}$ is defined.
Let us make a reparametrization and a coordinate change similar to the previous ones:
\begin{align}
    \label{eqn:xbeta}
    ε &= x(U(β)) - x(S(β)), &x_β &= x - x(S(β)).
\end{align}
In this notation,
\begin{align}
    \label{eqn:eps}
    x_β(S(β)) &= 0, &x_β(U(β)) &= ε.
\end{align}

Let $X(β)≔x_β(I'(β))$.
Then the \emph{connection equation}
\begin{align}
    \label{eqn:conn}
    Δ_{γ,β}^{n-1}(ε)&=X(β), &β&=(η, ε),
\end{align}
is equivalent to the fact that $v_β$ has an $n$-winding connection $L(β)I(β)$, see \autoref{fig:conn-eq}.
\begin{lemma}
    \label{lem:asym2}
    In assumption of this subsection, the connection equation has a unique solution $i_n(η)$ for any $n$ large enough defined on a neighborhood of $0$ non-depending on $n$.
    The functional sequence $i_n$ is monotonically decreasing and
    \begin{equation}
        \label{eqn:asym5}
        \log(-\log i_n(η))=-n\log λ (η)+O(1).
    \end{equation}
    The upper bound for the remainder term in the right hand side is uniform in $η$.
\end{lemma}
This lemma is proved in the rest of this section.

\subsection{Local families of correspondence maps of hyperbolic saddles}
\label{sub:saddle}
Here we recall a definition of the correspondence (Dulac) map of a hyperbolic saddle.
\autoref{lem:corr-map} stated below claims that for a local family of such saddles the Dulac map behaves like $x^{λ(β)}$, where $λ(β)$ is the parameter depending characteristic value of the saddle.
In case of a parameter depending linear hyperbolic saddle,
$\dot x = x$; $\dot y = -λ(β)y$,
we have: $Δ_β(x)=x^{λ(β)}$, and the lemma is trivial.
Suppose now that $λ(0)$ is irrational, and $v_β$ is $C^∞$ smooth.
Then the corresponding local family is non-resonant.
Smooth orbital normal form for this family is linear \cite{IYa1}.
In this case the lemma is trivial again. The only non-trivial case is $λ(0)∈ℚ$.
The lemma for both cases simultaneously is proved at the end of this section.

Consider a local family of vector fields in the sphere with a hyperbolic saddle $L(β)$.
Let $λ(β)$ be the characteristic number of $L(β)$.
By a smooth parameter depending coordinate change we may put the saddle at the origin and two separatrixes of the saddle to positive coordinate rays.
A parameter depending \emph{correspondence} or \emph{Dulac map} $Δ_β$ is a map of a semitransversal $Γ^+$ with a vertex on the incoming separatrix to a semitransversal $Γ^-$ with a vertex on the outgoing one along the orbits of the vector field of the family;
the map is defined near the vertex of a semitransversal.
Consider smooth charts $x$, $y$ on the semitransversals $Γ^+$, $Γ^-$ that vanish at the vertexes.
These charts will be called \emph{natural}.

We shall formulate our estimates on $Δ_β$ using $g=O(f)$ and $g=Θ(f)$ notation.
Let $f$ and $g$ be two functions $ℝ_+→ℝ$.
We write
\begin{align*}
    g&=O(f) &&⇔ &∃C>0 \ ∃ε>0 \ ∀x∈(0, ε)&:|g(x)|≤C|f(x)|;\\
    g&=Θ(f) &&⇔ &∃0<c<C \ ∃ε>0 \ ∀x∈(0, ε)&: c|f(x)|≤|g(x)|≤C|f(x)|.
\end{align*}
For parameter depending maps $f_β$, $g_β$, $β∈(ℝ^k, 0)$, the equalities $g_β=O(f_β)$ and $g_β=Θ(f_β)$ mean the same relations as above with $c$, $C$ not depending on $β$.

In what follows, we use notation $D_x$, $D_ε$ for partial derivatives $\frac{∂}{∂x}$, $\frac{∂}{∂ε}$.
\begin{lemma}
    \label{lem:corr-map}
    Consider a local family $\set{v_β | β ∈ (ℝ^k,0)}$ of vector fields with a hyperbolic saddle $L_β$, having the correspondence map $Δ_β$.
    Let $ε$ be a component of $β$, $ε ∈(ℝ, 0)$.
    Then in any natural charts $x$, $y$ on the semitransversals $Γ^+$, $Γ^-$, the map $y=Δ_β(x)$ has the following properties:
    \begin{subequations}
        \label{eqn:corr-map}
        \begin{align}
            \label{eqn:corr:C0}
            Δ_β(x)&=Θ\left( x^{λ(β)} \right),\\
            \label{eqn:corr:C1-x}
            D_xΔ_β(x)&=Θ\left( x^{λ(β)-1} \right),\\
            \label{eqn:corr:C1-b}
            D_εΔ_β(x)&=O\left( x^{λ(β)}\log x \right).
        \end{align}
    \end{subequations}
    in some neighborhood of $0$ in $ℝ_+× ℝ^k$.
    All the estimates are uniform in $β$ for $β$ small enough.
\end{lemma}
The collection of these three properties will be referred to as property \eqref{eqn:corr-map}.
\begin{remark}
    \label{rem:chain}
    If two families of maps, $Δ_β$ and $\t Δ_β$, satisfy property \eqref{eqn:corr-map} with exponents $λ(β)$ and $μ(β)$ respectively, then their composition satisfies this property with the exponent $λ(β)μ(β)$.
    \ifdetails
    Indeed,
    \begin{align*}
        Δ_β∘\t Δ_β&=Θ\left(Θ\left(x^{μ(β)}\right)^{λ(β)}\right)\\
                  &=Θ\left(x^{λ(β)μ(β)}\right);\\
        D_x(Δ_β∘\t Δ_β)&=(D_xΔ_β)∘\t Δ_β×D_x\t Δ_β\\
                       &=Θ\left( Θ\left( x^{μ(β)} \right)^{λ(β)-1} \right)×Θ\left( x^{μ(β)-1} \right)\\
                       &=Θ\left( x^{λ(β)μ(β)-1} \right);\\
        D_β(Δ_β∘\t Δ_β)&=(D_xΔ_β)∘\t Δ_β×D_β\t Δ_β+(D_βΔ_β)∘\t Δ_β\\
                       &=Θ\left( Θ\left( x^{μ(β)} \right)^{λ(β)-1} \right)×O\left( x^{μ(β)}\log x \right)
                         +O\left( Θ\left( x^{μ(β)} \right)^{λ(β)} \right)×\log Θ\left( x^{μ(β)} \right)\\
                       &=O\left( x^{λ(β)μ(β)}\log x \right)+O\left( x^{λ(β)μ(β)}\log x \right)\\
                       &=O\left( x^{λ(β)μ(β)}\log x \right).
    \end{align*}
    \else
    This follows from the chain rule, and may be checked by a straightforward calculation that we skip.
    \fi
\end{remark}

\subsection{Poincaré maps for unperturbed hyperbolic polycycles}
Here we prove that the Poincaré map of a parameter depending hyperbolic polycycle satisfies property \eqref{eqn:corr-map}.

Consider a vector field $v$ with a hyperbolic polycycle $γ$ monodromic from inside or from outside.
Let $Δ_γ$ be its monodromy map corresponding to a semitransversal $Γ^+$ with a vertex $O ∈ γ$, and $λ (γ)$ be the characteristic number of the polycycle.
Let $Γ$ be a cross-section to $γ$ at $O$ such that $Γ⊃Γ^+$.
Let $x\colon (Γ,O)→(ℝ,0)$ be a smooth chart on $Γ$ positive on $Γ^+ \setminus O$.

\begin{corollary}
    \label{cor:P-polycycle-noparam}
    The monodromy map $Δ_γ$ satisfies relations~\eqref{eqn:corr:C0}, \eqref{eqn:corr:C1-x} with $λ(β)$ replaced by $λ(γ)$.
\end{corollary}
\begin{remark}
    \label{rem:wind}
    This corollary, together with inequalities \eqref{eqn:char} (resp., \eqref{eqn:wind}), implies that $O$ is an attracting fixed point for the Poincaré map $Δ_γ$ (resp., $Δ_e$), and a repelling fixed point for the Poincaré map~$Δ_l$ (resp., $Δ_i$), where $Δ_γ$ and $Δ_l$ are the same as in \eqref{eqn:mon} (resp., $Δ_e$ and $Δ_i$ are the same as in \eqref{eqn:D-e-i}).
\end{remark}
\begin{proof}
    The proof follows from \autoref{rem:chain}.
    More details are given in the proof of the next corollary that works for the present one as well.
\end{proof}
Let us turn back to the local family $ℰ$ from \autoref{sub:gen} and consider vector fields with the polycycle $γ(η)$.
Let $λ(γ(η))$ be its characteristic number.
Put $O(η)=γ(η)∩Γ$.
Denote by $Δ_{γ,η}$ the Poincaré map corresponding to $γ (η)$ written in the chart $x_η=x-x(O(η))$.
\begin{corollary}
    \label{cor:P-polycycle}
    The monodromy map $Δ_{γ,η}$ satisfies relations~\eqref{eqn:corr:C0}, \eqref{eqn:corr:C1-x} with $x$, $λ(β)$ replaced with $x_η$ and $λ(γ(η))$.
\end{corollary}
\begin{proof}
    The map $Δ_{γ,η}$ is a composition of the correspondence maps for hyperbolic saddles, the vertexes of the polycycle.
    Each correspondence map satisfies~\eqref{eqn:corr:C0} and \eqref{eqn:corr:C1-x} due to \autoref{lem:corr-map}.
    Thus the composition of these maps satisfies~\eqref{eqn:corr:C0} and \eqref{eqn:corr:C1-x} with the exponent $λ(β)$ equal to the product of the exponents for individual maps, see \autoref{rem:chain}.
\end{proof}

\subsection{Poincaré maps for slightly perturbed hyperbolic polycycles}\label{sub:poipert}

Let us now consider vector fields of the family $ℰ$ for $ε \neq 0$.
In the notation of \autoref{sub:gen}, let $Λ (β)$ be the product of the characteristic values of all the saddles met on the way from $M(β)$ to $L(β)$
along the chain of unbroken saddle connections, including $M(β)$ and $L(β)$.
By definition, $Λ (η, 0) = λ (η)$.

\begin{corollary}
    \label{cor:2}
    In the settings above, in the chart $x = x_β$, see~\eqref{eqn:xbeta}, we have
    \begin{equation}
        \label{eqn:D-plus-eps}
        Δ_{γ, β}(x)=ε+\wt Δ_{γ, β}(x),
    \end{equation}
    where $\wt Δ_{γ, β}$ satisfies~\eqref{eqn:corr-map}, with $λ(β)$ replaced by $Λ(β)$.
\end{corollary}

\begin{proof}
    The corollary follows from~\eqref{eqn:eps} and the observation that the map $\wt Δ_{γ, β}$ still is a product of correspondence maps for hyperbolic saddles.
    The rest of the proof proceeds as in \autoref{cor:P-polycycle}.
\end{proof}

\subsection{Simple bounds for the monodromy map \texorpdfstring{$Δ_{γ, β}$}{Δ\_\{γ, β\}}}
In the next subsection we complete the proof of the Asymptotic \autoref{lem:asym2}.
For this we have to study the iterates of the map $Δ_{γ,β}$.
In this subsection we compare this map with a simpler one, namely, with a monomial $C x^λ$ for some $λ$ and $C$.
The iterates of the latter map are easy to calculate.
It is important that the comparison works in the domain
\begin{equation}
    \label{eqn:u}
    U=\set{(x, β)|x≥ε>0}
\end{equation}
for $x$ and $β = (η, ε)$ small.
The lower bound for $x$ is motivated by the equality $Δ_{γ,β}(0)=ε$.
\begin{lemma}
    [Comparison Lemma]
    \label{lem:estim}
    Let $Λ(0) < 1$,
    $U$ be the same as in \eqref{eqn:u}.
    Then in the chart $x = x_β$ we have
    \begin{subequations}
        \label{eqn:broken-D}
        \begin{align}
            \label{eqn:broken-D:C0}
            Δ_{γ, β}(x)&=Θ\left( x^{Λ(η, 0)} \right)\\
            \label{eqn:broken-D:C1-x}
            D_xΔ_{γ, β}&>2;\\
            \label{eqn:broken-D:C1-e}
            D_εΔ_{γ, β}&>0
        \end{align}
    \end{subequations}
    for $(x, β )∈U$ small.
    Recall that \eqref{eqn:broken-D:C0} is, by definition, uniform in $β$.
\end{lemma}
\begin{remark}
    Instead of 2, we could take any constant greater than 1 in relation~\eqref{eqn:broken-D:C1-x}: below we will prove that $D_xΔ_{γ, β} →∞$ as $x →0$ uniformly in $β$.
\end{remark}
\begin{proof}
    Let us apply $D_ε$ to both sides of \eqref{eqn:D-plus-eps} from \autoref{cor:2}, then use \eqref{eqn:corr:C1-b}.
    We have:
    \[
        D_εΔ_{γ,β}(x)=1+D_ε\t Δ_{γ,β}(x)=1+O\left( x^{Λ(β)}\log x \right)=1+o(1).
    \]
    This implies \eqref{eqn:broken-D:C1-e}.
    Similarly,
    \[
        D_xΔ_{γ,β}(x)=D_x\t Δ_{γ,β}(x)=Θ\left( x^{Λ(β)-1} \right)→∞\text{ as $x→0$ uniformly in $β$},
    \]
    since $Λ(0)<1$, and $β$ is small.
    Thus we proved \eqref{eqn:broken-D:C1-x}.

    Finally, let us prove \eqref{eqn:broken-D:C0}.
    Due to \eqref{eqn:D-plus-eps} and \eqref{eqn:corr:C0},
    \[
        Δ_{γ, β}(x)=ε+Θ\left( x^{Λ(β)} \right),\quad β=(η, ε).
    \]

    For the first summand we have
    \[
        ε≤x=o\left( x^{Λ(η, 0)} \right),
    \]
    since $0<Λ(η, 0)<c<1$ for $η$ small enough.
    For the second summand, we need to prove that
    \[
        x^{Λ(η, ε)}=Θ\left(x^{Λ(η, 0)}\right) \text{ as $x→0$ inside $U$,}
    \]
    or equivalently,
    \[
        x^{Λ(η, ε)-Λ(η, 0)}=Θ(1) \text{ as $x→0$ inside $U$.}
    \]
    Since $Λ$ is a smooth function, we have
    \[
        x^{Λ(η, ε)-Λ(η, 0)}=x^{O(ε)}=\left(x^ε\right)^{O(1)}.
    \]
    Note that for $(x, β)∈U$, $x<1$, we have $ε^ε≤x^ε≤1$, and $ε^ε→1$ as $ε→0$.
    Therefore,
    \[
        x^{Λ(η, ε)-Λ(η, 0)}=(1+o(1))^{O(1)}=1+o(1).
    \]
    This implies~\eqref{eqn:broken-D:C0}.
\end{proof}

\subsection{Sparkling saddle connections}
In this section we complete the proof of \autoref{lem:asym2}, that is, prove the existence, uniqueness and monotonicity in $n$ of the solutions of the connection equation~\eqref{eqn:conn}.

\subsubsection{Existence}
The arguments below are based on the Comparison \autoref{lem:estim}.
Namely we replace the Poincaré map in the connection equation by a smaller (larger) map, whose iterates may be easily studied and for which the connection equation may be easily solved.
These solutions provide the estimates for the actual solutions of the connection equation.

Let us pass to the detailed proof.
Denote still by $U$ a small subdomain of $\set{ (x,β )| 0 < ε \le x}$ where all the estimates of this lemma hold.
More precisely, let
\begin{align*}
    K &= \set{ β = (η , ε )| \relax|η| < r, \ ε ∈ [0,ε_0]},&
    U &= \set{ (x, β )| β ∈ K, \ ε ≤ x ≤ r }
\end{align*}
for $ε_0, r$ small.
By~\eqref{eqn:broken-D:C0}, there exists $C > 0$ such that
\[
    e^{-C}x^{Λ (η ,0)} < Δ_{γ,β} (x ) < e^Cx^{Λ (η ,0)},
\]
where $β = (η , ε )$, $(x,β )∈U$.
Fix a small $η $ and denote for simplicity
\[
    Λ (η ,0) = λ , \ f_-(x) = e^{-C}x^λ, \ f_+(x) = e^Cx^λ .
\]
The following estimates are obviously uniform in $η $ small;
for this reason we omit the dependence of $λ$ on $η$ in notation.
Each of the maps $f_±$ has two fixed points: $0$ and $x_± = \exp{\left(± \frac {C}{1-λ }\right)}$.
Note that $x_- < 1 < x_+=x_-^{-1}$.
On the segment $[0,x_-]$ the map $f_-$ pushes all points from the repeller $0$ to the attractor $x_-$.
In particular, all the negative iterates of $f_-$ are well-defined in $(0, x_-)$, hence for any $τ∈(0,x_-)$, and any $n$ the equation $f_-^n(x) = τ $ has a solution in $(0, x_-)$.
Let $U$ be chosen so that in $U$, $x < x_- < x_+$.

Let $m∈ℕ$, $τ∈(0, x_-)$ and $ε$ be such that $f^m_-(ε) = τ$, and $(ε, β)∈U$.
Then
\begin{equation}
    \label{eqn:estimp}
    Δ^m_{γ,β} (ε) > τ.
\end{equation}

The right hand side of the connection equation~\eqref{eqn:conn} does not in general belong to $(0, x_-)$.
Let us replace equation~\eqref{eqn:conn} by an equivalent one with the right hand side in $(0,x_-)$.
For this let
\[
    T(β ) = x_β (Δ_{γ, β} ^{-a}(I'(β ) )).
\]

Note that for $a$ sufficiently large, $T(η ,0)$ may be arbitrary small because the separatrix of $I(η ,0)$ approaches the connection $l$ in the negative time.
Hence, for $r$ in the definition of $K$ sufficiently small, $a$ may be chosen so large that for any $β∈K$, $T(β ) < x_-$.
Connection equation~\eqref{eqn:conn} is equivalent to
\begin{equation}
    \label{eqn:conn2}
    Δ_{γ,β} ^{n-a-1}(ε ) = T(β ), \ β∈K.
\end{equation}
Let $m = n - a - 1$.
Take
\begin{equation}
    \label{eqn:psim}
    ψ_{m,η }(ε) = Δ^{m}_{γ,β} (ε) - T( β ), \ β = (η , ε).
\end{equation}
Let $ε^+_m$ be the solution of the equation
\begin{equation}
    \label{eqn:estp}
    f^m_-(ε^+_m) = τ , \ τ = \max_K T
\end{equation}
It exists because $τ < x_-$.
Then $ψ_{m,η}(ε^+_m) > 0$, by~\eqref{eqn:estimp}.
Let $ε^-_m$ be the solution of the equation
\begin{equation}
    \label{eqn:estm}
    f^m_+(ε^-_m) = t, \ t = \min_K T.
\end{equation}
It exists because $t < x_- < x_+$.
Then $φ_{m,η }(ε^-_m) < 0$.
By the Intermediate Value Theorem, equation $φ_{m,η }(ε ) = 0$ which is equivalent to~\eqref{eqn:conn} has a solution $ε_m$, and
\begin{equation}
    \label{eqn:estd}
    ε^-_m < ε_m < ε^+_m.
\end{equation}

\subsubsection{Uniqueness}
Let us prove that for $n$ large enough solution for~\eqref{eqn:conn} is unique.
Suppose that $n$ is so large that we can replace \eqref{eqn:conn} by \eqref{eqn:conn2}.
Then it is sufficient to prove that $D_ε ψ_{m,η}(ε)>0$ for large $m=n-a-1$ and sufficiently small $ε$.
We have:
\[
    D_ε \left(Δ_{γ,β}^m(ε)\right)=
        ∑_{l=0}^{m-1}\left(D_x Δ_{γ,β}^l\right)\left(Δ_{γ,β}^{m-l}(ε)\right)×
        D_ε Δ_{γ,β}\left(Δ_{γ,β}^{m-l-1}(ε)\right)+\left(D_x Δ_{γ,β}^m\right)(ε).
\]
Due to~\eqref{eqn:broken-D:C1-x} and~\eqref{eqn:broken-D:C1-e}, all the summands above are positive, and the last one is greater than $2^m$.
Therefore, for $m$ large enough and $ε>0$ small enough, $D_ε \left(Δ_{γ,β}^m(ε)\right)$ is greater than $|D_ε T|$, hence
\begin{equation}
    \label{eqn:mone}
    D_ε ψ_{m,η}(ε)>0.
\end{equation}

\subsubsection{Monotonicity}
Let us prove that for $m$ large enough, $ε_m (η) < ε_{m-1} (η)$.
By definition, for $β= (η, ε_m)$,
\[
    Δ_{γ,β}^{m}(ε_m) = T(β).
\]
Inequality~\eqref{eqn:broken-D:C0} implies that for $ε > 0$, $Δ_{γ,β} (x) > x$.
Hence,
\[
    Δ_{γ,β}^{m-1}(ε_m) < T(β).
\]
Let $ψ_{m-1,η}$ be the same as in~\eqref{eqn:psim}, with $m$ replaced by $m-1$.
The latter inequality implies: $ψ_{m-1,η }(ε_{m}) <0$.
By~\eqref{eqn:mone}, $ψ_{m-1,η }(ε)$ monotonically increases in $ε$.
By definition, $ψ_{m-1,η}(ε_{m-1}) = 0$.
Hence, $ε_{m-1} > ε_m$.

\subsubsection{Estimates}
Equations~\eqref{eqn:estp} and~\eqref{eqn:estm} may be easily solved;
together with~\eqref{eqn:estd}, the formulas for solutions imply \eqref{eqn:asym5}.
Recall that
\[
    f_-(x) = e^{-C}x^λ,\quad
    x_- = e^{-\frac {C}{1-λ }} < 1.
\]
Note that
\[
    f_-^m(x) = C_mx^{λ^m},\quad
    C_m = e^{-C\frac {1-λ^m}{1-λ }} \searrow x_- \mbox{ as } m→∞.
\]
For any $τ < x_-$, equation $f_-^m(x)=τ$ is equivalent to $x^{λ^m} = τ C^{-1}_m$.
Note that $τ C_m^{-1} \nearrow τ x_-^{-1} < 1$ as $m→∞$.
Then, for the solution $ε_m^+$ of \eqref{eqn:estp}
\[
    \log (-\log ε^+_m) = -m\log λ + τ_0 + o(1),\quad
    τ_0 = \log (-\log τ x_-^{-1}).
\]
Similarly, $f_+(x) = e^Cx^λ$, $f^m_+(x) = C_m^{-1}x^{λ^m}$.
Then, for the solution $x = ε_m^-$ of \eqref{eqn:estm},
\[
    \log (-\log ε_m^-) = -m\log λ + t_0 + o(1),\quad
    t_0 = \log (-\log τ x_-).
\]
Note that $t_0 > τ_0$, because $x_- < 1$.
By \eqref{eqn:estd}
\[
    -m\log λ + τ_0 + o(1) < \log (-\log ε_m) < -m\log λ + t_0 + o(1).
\]
This proves~\eqref{eqn:asym5} and, together with it, \autoref{lem:asym2}.

\subsection{Proof of asymptotic lemmas}

In this section we deduce Lemmas~\ref{lem:asym}, \ref{lem:asym1} from \autoref{lem:asym2}.
\autoref{lem:asym} follows from \autoref{lem:asym1}, so we prove the latter one only.

In statement~\eqref{eqn:im1}, we deal with a polycycle $γ$ such that $λ(γ)<1$.
Hence we can apply \autoref{lem:asym2} and obtain~\eqref{eqn:im1}.
In order to prove the other half of \autoref{lem:asym1}, relation \eqref{eqn:en1}, it is enough to reverse the time, i.e., to replace the vector field $v$ by $-v$.

Let us now prove~\eqref{eqn:en1} in more details.
Consider a family $ℰ$ from \autoref{lem:asym1}.
Let $γ(η)$ be the monodromic polycycle of the vector field $v_{(η, 0)}$.
By assumption~\eqref{eqn:char}, $λ (η) > 1$.
By definition of the family $ℰ$, there exists a hyperbolic saddle $E$ whose separatrix winds to $γ_e(η)$ in the positive time.
The family $ℰ $ from \autoref{lem:asym1} is a particular case of the family~\eqref{eqn:famee} with the only difference:
$λ (γ) > 1$ instead of $λ(γ) < 1$.
Let us now reverse the time.
We will get a vector field $-v_{(η,0)}$ with the same polycycle, but with an opposite orientation.
Denote it by $γ^-(η)$.
The characteristic numbers of the saddles, under the time reversal, are replaced by the reciprocal ones.
Hence, $λ (γ^-(η)) = λ^{-1} (γ_e(η)) <1 $.
The modified family $ℰ $ is a particular case of the family~\eqref{eqn:famee}.
\autoref{lem:asym2} is now applicable;
it implies:
\[
    \log (-\log e_n(η)) = - n \log (λ^{-1} (η)) + O(1),
\]
This relation is equivalent to~\eqref{eqn:en1}.
This completes the proof of \autoref{lem:asym1}, hence \autoref{lem:asym}, modulo \autoref{lem:corr-map}.

\subsection{Correspondence maps of hyperbolic saddles: proof of \autoref{lem:corr-map}}
\paragraph{Preliminary considerations}

Here we prove \autoref{lem:corr-map}.
Let $\set{v_β}$ be a family of vector fields described in \autoref{lem:corr-map};
$β$ is now the same as in \autoref{sub:saddle}.
Clearly, the assertion of the lemma depends neither on the choice of $Γ^+$ and $Γ^-$, nor on the choice of natural charts $x$ and $y$.

Let us choose some coordinates $(x_β, y_β)$ near $L(β)$ such that $S(β)=\set{x_β=0}$ and $U(β)=\set{y_β=0}$.
Due to Hadamard−Perron Theorem, we may and will assume that $x_β$ and $y_β$ are $C^3$-smooth functions of the original coordinates and $β$.
The differential of this coordinate change is $C^2$-smooth, thus $\dot x_β$ and $\dot y_β$ are $C^2$-smooth functions of $x_β$, $y_β$.
For simplicity, we write $(x,y)$ instead of $(x_β, y_β)$.
Since $\dot x(0, y)=0$ and $\dot y(x, 0)=0$, due to Hadamard Lemma, we have
\begin{align*}
    \dot x&=xf_1(x, y, β);&
    \dot y&=-yf_2(x, y, β),
\end{align*}
where $f_1$ and $f_2$ are $C^1$-smooth functions.
This vector field has the same correspondence map as the vector field
\begin{subequations}
    \label{eqn:orbital}
    \begin{align}
        \label{eqn:orbital-x}
        \dot x&=x;\\
        \label{eqn:orbital-y}
        \dot y&=-yg(x, y, β),
    \end{align}
\end{subequations}
where $g(x, y, β)=\frac{f_2(x, y, β)}{f_1(x, y, β)}$ is a $C^1$-smooth function as well.
Clearly, $g(0, 0, β)=λ(β)$.
Choose a~neighborhood ${U⊂ℝ^2× ℝ^k}$ of $(L(0),0)$ such that in $U$
\begin{equation}
    \label{ineq:g}
    \frac 12λ(0)<g(x, y, β)<2λ(0).
\end{equation}
Choose $Γ^+$ and $Γ^-$ so that all trajectories of $v_β$ going from $Γ^+$ to $Γ^-$ stay in $U$.
After a rescaling, we may and will assume that $Γ^+=\set{y=1}$ and $Γ^-=\set{x=1}$.
We shall prove \eqref{eqn:corr-map} for the restrictions of $x$ and $y$ to $Γ^+$ and $Γ^-$, respectively.

\paragraph{Estimate of $Δ_β$}
This estimate follows the proof of Lemma 1 in \cite{GK07}, see also \cite[Section 9.3]{HW}.

Fix a small positive $x_0$.
Consider a trajectory $(x,y)(t)$ of \eqref{eqn:orbital} starting at $(x_0, 1)$.
In our notation, we skip the indication of the dependence of the solution on the parameter $β $.
Due to \eqref{eqn:orbital-x}, we have $x(t)=x_0e^t$, hence this trajectory arrives to $Γ^-$ at $T=-\log x_0$.
Next, \eqref{ineq:g} implies that $y(t)≤e^{-λ(0)t/2}$, but we need a sharper estimate.
Note that for $0≤t≤T$ we have
\begin{equation}
    \label{eqn:O-x-y}
    ∫_0^tO(x(τ))+O(y(τ))\,dτ=∫_0^tO\left(e^{τ-T}\right)+O\left(e^{-λ(0)τ/2}\right)\,dτ=O(1).
\end{equation}
Therefore, for $0≤t≤T$ we have
\begin{align*}
    \log y(t)&=∫_0^t\frac{\dot y(τ)}{y(τ)}\,dτ=-∫_0^tg(x(τ), y(τ), β)\,dτ\\
             &=-∫_0^t\left( λ(β)+O(x(τ))+O(y(τ))\right)dt = -λ(β)t+O(1).
\end{align*}
Thus
\begin{equation}
    \label{eqn:y-t}
    y(t)=Θ\left( e^{-λ(β)t} \right).
\end{equation}

From now on replace $x_0 $ by $x$.
For $t=T = - \log x$, equation~\eqref{eqn:y-t} implies~\eqref{eqn:corr:C0}, because $y(T) = Δ_β (x)$, $e^{-λ(β)T} = x^{λ(β)}$.

\paragraph{Decomposition of $Δ_β$}
In order to estimate the derivatives of the germ of $Δ_β$ at $x$, let us decompose it in the following way.
Consider an auxiliary local cross-section ${Γ^x}⊂\set{(x, y), \ y∈ (ℝ,1)}$; it passes through the point $(x,1)$ and is orthogonal to $Γ^+$.
Equip $Γ^x$ with the chart $y$.
Let
\begin{align*}
    Δ_β^+&\colon (Γ^+, x)→(Γ^x, 1),\\
    Δ_β^{x}&\colon (Γ^x, 1)→(Γ^-, Δ_β(x));
\end{align*}
be the correspondence maps along the orbits of $v_β$, same as correspondence maps for~\eqref{eqn:orbital-x}, \eqref{eqn:orbital-y}.
Then
\[
    Δ_β=Δ_β^x∘Δ_β^+.
\]

\begin{figure}
    \centering
    \includegraphics[scale=0.5]{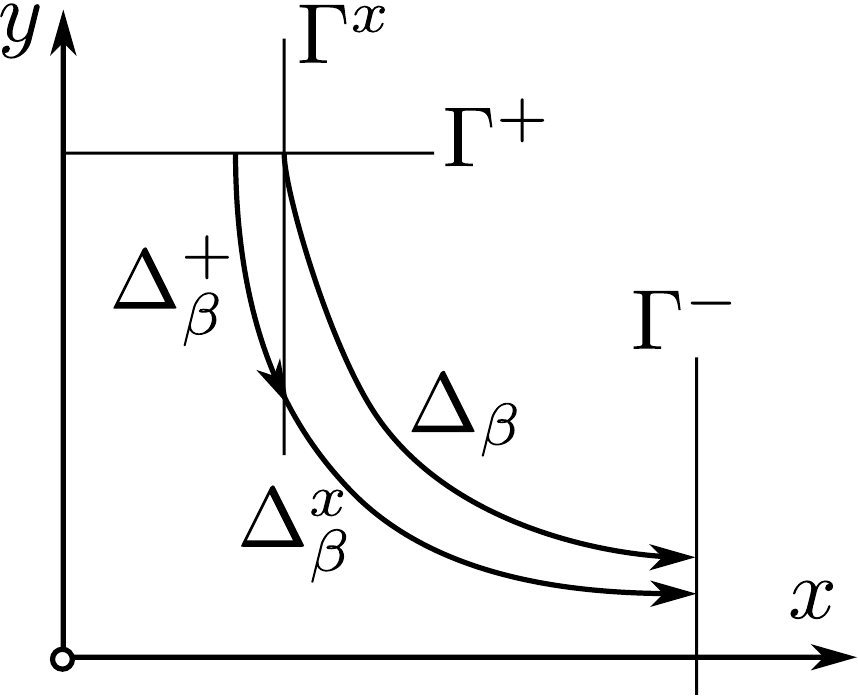}
    \caption{A saddle, some cross-sections and correspondence maps}
    \label{fig:Delta_x}
\end{figure}

\paragraph{Estimate of $D_xΔ_β(x)$}
Note that $D_xΔ_β^+(x)=\frac{g(x, 1, β)}{x}=Θ(x^{-1})$, hence it is enough to show that $D_yΔ_β^x(y)|_{y=1}=Θ\left( x^{λ(β)} \right)$.

Recall that $x(t)=x(0)e^t$.
Equation~\eqref{eqn:orbital-y} with $x$ replaced by $x(t)=x(0)e^t$, is a one-dimensional non autonomous equation.
Let $y(t, β, a)$ be a solution of this equation with the initial condition $y(0, β, a) = a$;
as before, $β = (η, ε)$.
Denote by $z(t)$ and $w(t)$ the derivatives of this solution with respect to the initial conditions and the parameter $ε$:
\[
    z(t) = D_ay(t, β, a)|_{a = 1},
    \ w(t) = D_ε y(t, β, 1),
    \ ε∈(ℝ, 0).
\]
This implies: $z(0) = 1$, $w(0) = 0$.
Denote for simplicity $y(t, β,1) = y(t)$;
we skip in this notation the dependence on $β$ that actually takes place.
Note that
\[
    Δ_β (x) = y(T),
    \ T = - \log x;
    \ D_y Δ_β^x (1) = z (T),
    \ D_ε Δ_β^x (1) = w (T).
\]
The variational equation for $z(t)$ has the form:
\[
    \dot z = - A(t)z, \ A(t) = (D_y(yg))(x(t), y(t), β).
\]
Hence, $A(t) = g(x(t), y(t), β) + y D_y g(x(t), y(t), β) = λ (β) +O(x(t)) +O(y(t))$.
Therefore, by~\eqref{eqn:O-x-y},
\[
    \log z(t) = -\int_0^t A(τ) dτ = - λ (β) t + O(1);
    \ \log z(T) = λ (β) \log x + O(1).
\]
Hence, $(D_yΔ_β^x)(1)=z(T) = Θ\left( x^{λ(β)} \right)$, thus $D_xΔ_β(x)=Θ\left( x^{λ(β)-1} \right)$.
This implies \eqref{eqn:corr:C1-x}.

\paragraph{Estimate of $D_εΔ_β(x) = w(T)$}
The equation of variations for \eqref{eqn:orbital-y} with respect to the parameter takes the form
\begin{align*}
    \dot w &= -A(t) w - B(t),&
    B(t) &=\left( D_ε (yg)\right) (x(t), y(t), β),&
    w(0) &= 0,
\end{align*}
$A(t)$ is the same as above.
We have:
\begin{align*}
    w(t) &= z(t) C(t),&
    \dot C &= - B(t) z^{-1} (t),&
    w(0) &= 0.
\end{align*}
Hence,
\[
    w(T) = - z(T) \int_0^T B(t) z^{-1} (t)\,dt.
\]
We have: $B(t) = \left( D_ε (yg)\right) (x(t), y(t), β) = -y O(1)$, $z^{-1} (t) = O(1) \exp \left( λ (β) t \right)$.
Due to \eqref{eqn:y-t}, $\exp \left( λ (β) t \right)=Θ(y(t)^{-1})$, hence $z^{-1}(t)=O(y(t)^{-1})$.
Thus
\[
    D_ε Δ_β(x)=w(T) = Θ\left( x^{λ(β)} \right)∫_0^T O(1)\,dt=O\left( x^{λ(β)}\log x \right).
\]
This implies \eqref{eqn:corr:C1-b}.

\autoref{lem:corr-map} is proved.
It implies \autoref{lem:asym2}.
The proof of the main results of this paper is now completed.

\section{New perspective}\label{sec:perspective}

\subsection{No structural stability in codimension three}

\autoref{thm:unst1} implies that there is an open set in the space of three parameter families of vector fields that are not structurally stable.
Namely, consider two vector fields $v_0$, $\t v_0$ and their unfoldings that satisfy the assumptions of \autoref{sec:family}.
If the ratios of the logarithms of the characteristic numbers of the saddles $L$ and $M$ for these vector fields are different, then the two unfoldings are not moderately topologically equivalent, however close they are.

\subsection{No versal families whose dimension equals the codimension of the degeneracy: a conjecture}
In early 70's Arnold suggested a new approach that revolutionized the bifurcation theory.
It was based on a concept of \emph{versal families}.
The original definition of these families may be found in \cite{A}, and the final one in \cite{AAIS}.
Without reproducing this definition we will mention only that versal families are special unfoldings of a degeneracy of a certain class that contain a very concentrated information about the bifurcations in arbitrary local families that unfold the degeneracies of this class.
Since early 70's many versal families were investigated.
It so happened that their dimension was always equal to the codimension of the degeneracy, though this is not required by the definition.
Moreover, two different versal deformations of the same vector field are weakly (and even moderately) topologically equivalent.
This equivalence holds for all versal deformations studied up to now.
This is a folklore fact, not written anywhere.

\autoref{thm:func-unst} implies that two generic $6$-parameter local families passing through the same vector field of class $\bfT_{2,1}$, are not topologically equivalent.
Indeed, consider two generic unfoldings $V$ and $W$ of the same vector field of class $\bfT_{2,1}$.
These families belong to an open subset of $\mcV_6(S^2)$ mentioned in \autoref{thm:func-unst}.
By this theorem, they may have different functional invariants $f: (ℝ_+,a)→(ℝ_+,b)$, $g: (ℝ,a)→(ℝ_+,b)$, because any such germ may be realized as an invariant of such a local family.
Hence, the families $V$ and $W$ are not equivalent.

This gives a strong evidence to the following

\begin{conjecture}
    There are no versal $6$-parameter local families that unfold generic vector fields of class $\bfT_{2,1}$.
    Moreover, for any $k ≥ 6$ there exists an open set of vector fields for which there are no $k$-parameter versal families.
\end{conjecture}

The second author has a strategy of the proof of this conjecture.
Note that our arguments imply no functional invariants for \emph{seven}-parametric unfoldings of vector fields of class $\bfT_{2,1}$.
The problem of existence of versal families whose dimension is higher than the codimension of the degeneracy is still open.

\begin{problem}
    Let $\bfM⊂\Vect(S^2)$ be a Banach submanifold of finite codimension.
    Is it true that for a generic $v∈\bfM$ and $k≥\codim\bfM$ large enough, there exists a $k$-parameter versal deformation of $v$?
\end{problem}

The authors expect that there exist classes $\bfM$ for which the answer is negative.

\subsection{Good, bad and ugly families of vector fields}
The space of all finite-parameter families of vector fields on the sphere may be split in three classes, each one more complicated than the previous class.

\emph{Good families} are moderately structurally stable ones.

\begin{conjecture}
    \label{conj:1param}
    Generic one-parameter families of vector fields in the two sphere are good.
\end{conjecture}

More difficult is

\begin{conjecture}
    Generic two-parameter families of vector fields in the two sphere are good.
\end{conjecture}

We do not expect that the proof of this conjecture is either easy or short.
Plausibly, it requires a topological classification of all different classes of two parameter families.
This classification is discussed below.

\emph{Bad families} are those whose moderate topological classification has numeric invariants.
These families are not structurally stable.

\emph{Ugly families} are those whose moderate topological classification has functional invariants.

Study of the boundaries between good, bad and ugly families is a challenging problem.
A particular statement is:

\begin{problem}
    Distinguish structurally unstable generic three parameter families from the structurally stable ones.
\end{problem}

In this paper we proved that bad three-parameter families exist.
No doubt that good three-parameter families exist too.

An interesting problem is to construct more examples of bad three-parameter families.

\begin{problem}
    What is the smallest number of parameters for which ugly families exist?
\end{problem}

\autoref{thm:fuinv} implies that this number is no greater than $6$.
We expect that the actual answer is even smaller.

\subsection{Classification problems}

\begin{problem}
    Classify all generic one-parameter local families of vector fields on the two-sphere up to moderate topological equivalence.
\end{problem}

Preliminary steps of this classification, in particular, the classification of all the possible degeneracies, are done in \cite{S74}.
A complete description, without a proof, is suggested in \cite{I}.
Together with his students N.~Solodovnikov and V.~Starichkova, the first author completes the justification of this description.
As a by product, this description implies \autoref{conj:1param}.
\begin{problem}
    Classify all generic two-parameter local families of vector fields on the two-sphere up to moderate topological equivalence.
\end{problem}

This problem is far from being solved.
The strategy may be the following.
In \cite{KS}, a complete list of polycycles that may occur in generic two and three parameter families was presented.
It looks natural to study “sparkling saddle suspensions” over these polycycles.
This construction may produce about two or three dozens of infinite series of two-parameter local families, each one corresponding to suspensions of sparkling saddle connections over the polycycles of codimension two from the list, or over two coexisting polycycles of codimension one.
Two families in a series differ by a number and a mutual location of the separatrixes of saddles that form sparkling saddle connections.
There may be other invariants.
The families in each series should be classified like it is done for one-parameter families.

Complete classification of generic three-parameter families looks like a very large, but not yet hopeless problem.
One of the first steps may be the following.

\begin{problem}
    Find ALL the topological invariants of the family described in \autoref{thm:unst1}.
\end{problem}

\subsection{Continuum of germs of bifurcation diagrams}

Arnold conjectured \cite{AAIS} that for any $k$ there exists but a finite number of pairwise topologically nonequivalent germs of bifurcation diagrams that may occur in generic $k$-parameter families.
(A bifurcation diagram is a subset of the base of the family that corresponds to structurally unstable vector fields.)
This conjecture was disproved in \cite{KS} where a countable number of nonequivalent germs of bifurcation diagrams in three-parameter families was constructed.
Recently it occurred that this effect may be observed even in two-parameter families \cite{I}.

\begin{conjecture}
    There exists an open set in the space of three-parameter families such that the set of pairwise topologically different germs of bifurcation diagrams that may occur in the families from this set has cardinality continuum.
\end{conjecture}

The first two authors have a strategy of the proof of this conjecture.

\subsection{Numeric and functional invariants in local and semilocal bifurcations }

Semilocal theory studies bifurcations in a neighborhood of an arbitrary polycycle; denote this polycycle by $γ$.
The phase space is now a germ of a neighborhood $(ℝ^2, γ)$.
Moderate topological equivalence of semilocal families on $(ℝ^2, γ) × (ℝ^k,0)$ is defined as in \autoref{sec:intro}.

\begin{problem}
    Are there numeric or functional invariants of moderate topological classification of semilocal families?
\end{problem}

\begin{problem}
    The same question about the local families: may numeric or functional invariants occur in the moderate topological classification of families of vector fields in a neighborhood of a singular point?
\end{problem}

We expect that “hidden” sparkling saddle connections may occur in the unfoldings of polycycles, even hyperbolic, and produce numeric and functional invariants.
The same expectation concerns unfoldings of complex singular points.
This is a realization of the following heuristic principle:
\begin{quote}
    \emph{All effects observed for global families in the plane may be observed for the local and semi-local ones, may be, with a greater number of parameters.}
\end{quote}

\paragraph{Acknowledgments}
The authors are grateful to Christian Bonatti, Anton Gorodetski and Alexei Klimenko for fruitful suggestions.

\printbibliography
\end{document}